%% file: Kautz-subdigraphs.tex
\documentclass[a4paper]{article}

\usepackage[pdftex]{graphicx}

\title{On Cyclic Kautz Digraphs
\thanks{K. B\"{o}hmov\'a is a recipient of the Google Europe Fellowship in Optimization Algorithms, and this research is
supported in part by this Google Fellowship. Research of C. Dalf\'{o} supported by projects MINECO MTM2011-28800-C02-01, and 2014SGR1147 from the Catalan Government. Research of C. Huemer supported by projects MINECO MTM2012-30951 and
Gen. Cat. DGR 2014SGR46. }}

\author{K. B\"{o}hmov\'a$^a$, C. Dalf\'{o}$^b$, C. Huemer$^b$
\\ \\
{\small $^a$ Dept. of Computer Science, ETH Z\"{u}rich, Switzerland} \\
{\small {\tt katerina.boehmova@inf.ethz.ch}}\\
{\small $^b$ Dept. de Matem\`atica, Universitat Polit\`ecnica de Catalunya, Barcelona, Catalonia}\\
{\small {\tt\{cristina.dalfo,clemens.huemer\}@upc.edu}}}

\date{}

\usepackage{epsfig,graphicx,a4wide}
\usepackage{amsfonts,amsmath,amsthm,amssymb}
\usepackage{url}
\usepackage{cases}
\newtheorem{theorem}{Theorem}[section]
\newtheorem{proposition}[theorem]{Proposition}
\newtheorem{lemma}[theorem]{Lemma}

\newtheorem{remark}[theorem]{Remark}

\begin{document}

\maketitle

\begin{abstract}
A prominent problem in Graph Theory is to find extremal graphs or digraphs with restrictions in their diameter, degree and number of vertices. Here we obtain a new family of digraphs with minimal diameter, that is, given the number of vertices and out-degree there is no other digraph with a smaller diameter. This new family is called modified cyclic digraphs $MCK(d,\ell)$ and it is derived from the Kautz digraphs $K(d,\ell)$. 

It is well-known that the Kautz digraphs $K(d,\ell)$ have the smallest diameter among all digraphs with their number of
vertices and degree. We define the cyclic Kautz digraphs
$CK(d,\ell)$, whose vertices are labeled by all possible sequences $a_1\ldots a_\ell$ of length $\ell$,
such that each character $a_i$ is chosen from an alphabet containing $d+1$ distinct symbols,
where the consecutive characters in the sequence are different (as in Kautz digraphs), and now also requiring that $a_1\neq a_\ell$.
The cyclic Kautz digraphs $CK(d,\ell)$ have arcs between vertices $a_1
a_2\ldots a_\ell$ and $a_2 \ldots a_\ell a_{\ell+1}$,
with $a_1\neq a_\ell$ and $a_2\neq a_{\ell+1}$.
Unlike in Kautz digraphs $K(d,\ell)$, any label of a vertex of
$CK(d,\ell)$ can be cyclically shifted to form again a label of
a vertex of $CK(d,\ell)$.

We give the main parameters of $CK(d,\ell)$: number of vertices, number of arcs, and diameter. Moreover, we construct the modified cyclic Kautz digraphs $MCK(d,\ell)$  to obtain the same diameter as in the Kautz digraphs, and we show that $MCK(d,\ell)$ are $d$-out-regular. Finally, we compute the number of vertices of the iterated line digraphs of $CK(d,\ell)$.
\end{abstract}


\section{Introduction}

Searching for graphs or digraphs with maximum number of vertices given maximum degree $\Delta$ and diameter $D$, or with minimum diameter given maximum degree $\Delta$ and number of vertices $N$ are two very prominent problems in Graph Theory. These problems are called the $(\Delta,D)$ problem and the $(\Delta,N)$ problem, respectively. See the comprehensive survey by Miller and \v{S}ir\'{a}\v{n}~\cite{MiSi13} for more information. 
In this paper, we obtain a new family of digraphs with minimal diameter, in the sense that given the number of vertices and out-degree there is no other digraph with a smaller diameter. This new family is called modified cyclic Kautz digraphs $MCK(d,\ell)$ and it is derived from the Kautz digraphs $K(d,\ell)$.

It is well-known that the Kautz digraphs $K(d,\ell)$, where $d$ is the degree, have vertices labeled by all possible sequences $a_1\ldots a_\ell$ of length $\ell$
with different consecutive symbols, $a_i\neq a_{i+1}$ for $i=1,\ldots,\ell-1$, from an alphabet $\Sigma$ of $d+1$ distinct symbols.
The Kautz digraphs $K(d,\ell)$ have arcs between vertices $a_1
a_2\ldots a_\ell$ and $a_2 \ldots a_\ell a_{\ell+1}$. See Figure~\ref{fig:Kautz-normals}.
Notice that between $a_1 a_2\ldots a_\ell$ and $a_2 \ldots a_\ell a_1$ there is not always an arc,
since $a_1$ and $a_\ell$ may be the same symbol, in which case $a_2 \ldots a_\ell a_1$ is not a vertex of $K(d,\ell)$.

In this paper, we define the cyclic Kautz digraphs
$CK(d,\ell)$ (see Figure~\ref{fig:donut}), where the labels of their vertices are defined
as the ones of the Kautz digraphs, with the additional requirement
that the first and the last symbol must also be different $(a_1\neq
a_\ell)$.
The cyclic Kautz digraphs $CK(d,\ell)$ have arcs between
vertices $a_1 a_2\ldots a_\ell$ and $a_2 \ldots a_\ell a_{\ell+1}$, with $a_i\neq a_{i+1}$, $a_1\neq a_\ell$ and $a_2\neq a_{\ell+1}$.
By this definition, we observe that the cyclic Kautz digraphs
$CK(d,\ell)$ are subdigraphs of the Kautz digraphs $K(d,\ell)$.
Unlike in Kautz digraphs $K(d,\ell)$, any label of a vertex of
$CK(d,\ell)$ can be cyclically shifted to form again a label of
a vertex of $CK(d,\ell)$.
We study some of the properties of cyclic Kautz digraphs.
Note that, in contrast to the Kautz digraphs, the cyclic Kautz
digraphs $CK(d,\ell)$ are not $d$-regular (neither $d$-out-regular). Therefore, for $CK(d,\ell)$ the meaning of $d$ is the size of the
alphabet minus one, and for $\ell>3$ $d$ also corresponds to the maximum out-degree of $CK(d,\ell)$. 

\begin{figure}[t]
\vskip-.75cm
    \begin{center}
  \includegraphics[width=10cm]{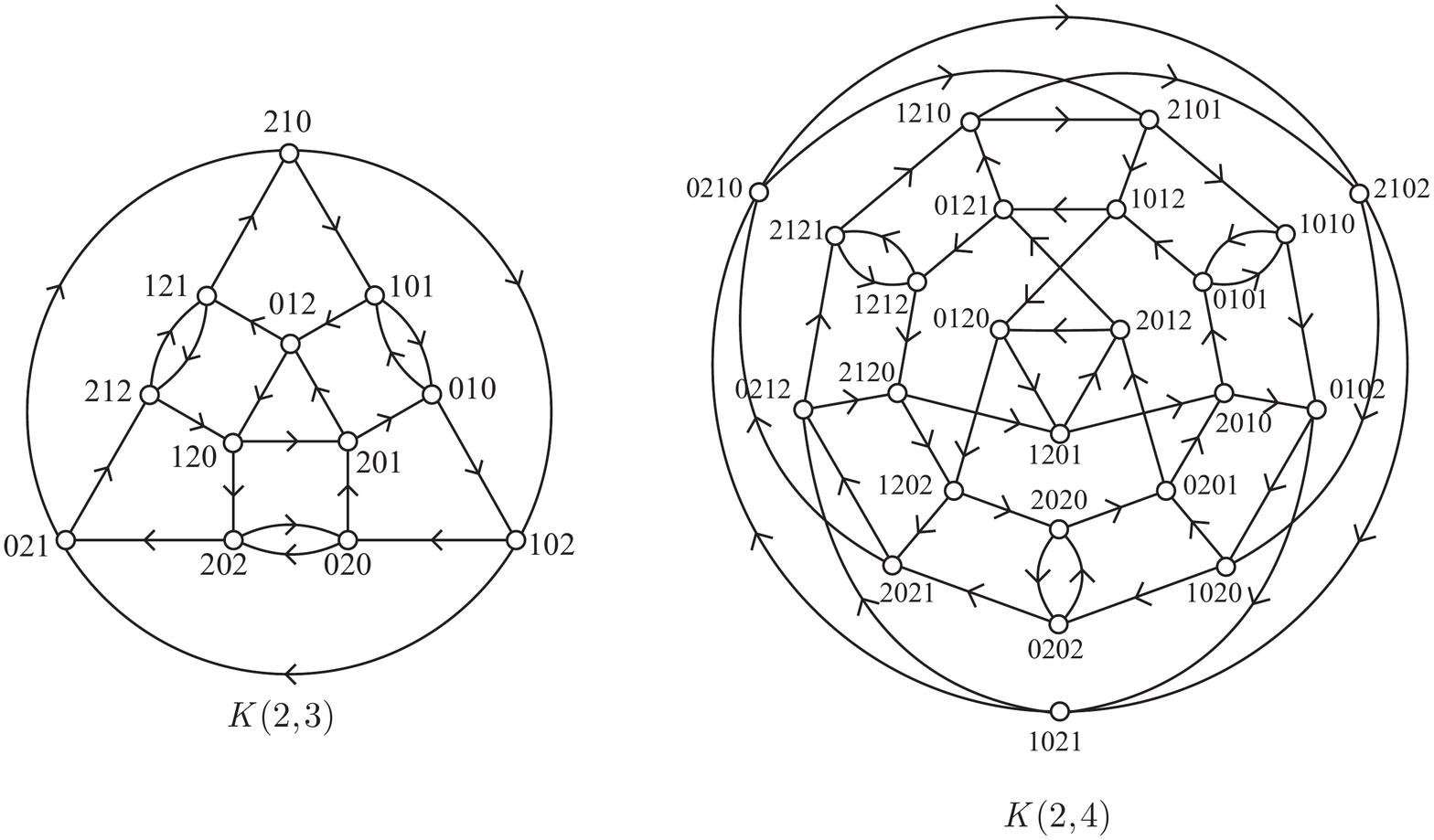}
  \end{center}
  \vskip-1.75cm
  \caption{The Kautz digraphs $K(2,3)$ and $K(2,4)$.}\label{fig:Kautz-normals}
\end{figure}

The cyclic Kautz digraphs $CK(d,\ell)$ are related to cyclic codes. A linear code $C$ of length $\ell$ is called cyclic if, for every codeword
$c=(c_1,\ldots,c_\ell)$, the codeword $(c_\ell,c_1,\ldots,c_{\ell-1})$ is also in $C$. 
This cyclic permutation allows to identify codewords with polynomials.
For more information about cyclic codes and coding theory, see Van Lint~\cite{vL92} (Chapter 6). With respect to other properties of the cyclic Kautz digraphs $CK(d,\ell)$, their number of vertices follows sequences that have several interpretations. For example, for $d=2$ (that is, 3 different symbols), the number of vertices follows the sequence $6,6,18,30,66,\ldots$ According to the On-Line Encyclopedia of Integer Sequences~\cite{Sl}, this is the sequence A092297, and it corresponds to the number of ways of 3-coloring a ring with $n$ zones joined like a pearl necklace. Moreover, dividing by 6 our sequence, we obtain $1,1,3,5,11,\ldots$, known as the Jacobsthal sequence (see~\cite{Sl}, A001045), with a lot of meanings, such as the number of ways to tile a 3x$(n-1)$ rectangle with 1x1 and 2x2 square tiles. For $d=3$ (4 different symbols) and $\ell=2,3,\ldots$, we get the sequence $12,24,84,240,732,\ldots$ An interpretation of this sequence corresponds to the number of closed walks of length $n$ in the complete graph $K_4$ (see~\cite{Sl}, A226493 and A218034). Dividing this sequence by 12, we obtain more interpretations (\cite{Sl}, A015518). 
%

Originally, the Kautz digraphs were introduced by Kautz~\cite{Ka68,Ka69} in 1968. They have many applications, for example, they are useful as network topologies for connecting processors. $K(d,\ell)$ have order $d^{\ell}$+$d^{{\ell} -1}$, where $d$ is equal to the cardinality of the alphabet minus one. They are $d$-regular, have small diameter $(D=\ell)$, and high connectivity. In fact, the Kautz digraphs have the smallest diameter among all digraphs with their number of vertices and degree. The Kautz digraphs $K(d,\ell)$ are related to the De Bruijn digraphs $B(d,\ell)$, which are defined in the same way as the Kautz digraphs, but without the restriction that adjacent symbols in the label of a vertex have to be distinct. Thus, the Kautz digraphs are induced subdigraphs of the De Bruijn digraphs, which were introduced by De Bruijn~\cite{dB1946} in 1946. The De Bruijn digraphs have order $d^{\ell}$, where the degree $d$ is equal to the cardinality of the alphabet, and they are also $d$-regular, have small diameter $(D=\ell)$, and high connectivity.
Another interesting property of the Kautz and De Bruijn digraphs is that they can be defined as iterated line digraphs of complete symmetric digraphs and complete symmetric digraphs with a loop on each vertex, respectively (see Fiol, Yebra and Alegre~\cite{FiYeAl84}). Note that the Kautz and De Bruijn digraphs are often referred to, as an abuse of language, as the Kautz and De Bruijn `graphs', which should not be confused with the underlying (undirected) Kautz and De Bruijn graphs.


It is known that the diameter of a line digraph $L(G)$ of a digraph $G$ is $D(L(G))= D(G)+1$, even if $G$ is a non-regular digraph with the exception of directed cycles (see Fiol, Yebra and Alegre~\cite{FiYeAl84}). Then, with the line digraph technique, we obtain digraphs with minimal diameter (and maximum connectivity) from a certain iteration. For this reason, we calculate the number of vertices of digraphs obtained with this technique for the cyclic Kautz digraphs. Computing the number of vertices of a $t$-iterated line digraph is easy for regular digraphs, but in the
case of non-regular digraphs, this is an interesting combinatorial
problem, which can be quite difficult to solve.
Fiol and Llad\'{o} defined in~\cite{FiLl92} the partial line digraph $PL(G)$ of a digraph $G$, where some (but not necessarily all) of the vertices in $G$ become arcs in $PL(G)$. For a comparison between the partial digraph technique and other construction techniques to obtain digraphs with minimum diameter see Miller, Slamin, Ryan and Baskoro~\cite{MiSlRyBa13}. Since these techniques are related to the degree/diameter problem, we also refer to the comprehensive survey of this problem by Miller and \v{S}ir\'{a}\v{n}~\cite{MiSi13}. We will use the partial line digraph technique to obtain the modified cyclic Kautz digraphs $MCK(d,\ell)$.
%

In this paper we make the following contributions. We give, in
Section~\ref{sec:parametres}, the
main parameters of the (newly defined) cyclic Kautz digraphs
$CK(d,\ell)$, that is, the number of vertices, number of arcs, and
diameter. Then, in Section~\ref{sec:modificats}, we construct the
modified cyclic Kautz digraphs $MCK(d,\ell)$ in order to obtain
digraphs with the same diameter as the Kautz digraphs, and we show
that $MCK(d,\ell)$ are $d$-out-regular. Finally, in
Section~\ref{sec:digrafs-linia}, we obtain the number of vertices of
the $t$-iterated line digraph of $CK(d,\ell)$ for $1\leq t\leq\ell-2$,
and for the case of $CK(d,4)$ for all values of $t$. For the
particular case of $CK(2,4)$, these numbers of vertices follow
a Fibonacci sequence.
Some of the proofs are in the Appendix, in order to make this paper more readable,
although all the main ideas are given in Sections~\ref{sec:parametres}--\ref{sec:digrafs-linia}.

We use the habitual notation for digraphs, that is, a \emph{digraph} $G=(V,E)$ consists of a (finite) set $V=V(G)$ of vertices and a set $E=E(G)$ of arcs (directed edges) between vertices of
$G$. As the initial and final vertices of an arc are not necessarily different, the digraphs may have \emph{loops} (arcs from a vertex to itself), but not \emph{multiple arcs}, that is, there is at most one arc from each vertex to any other. If $a=(u,v)$ is an arc between vertices $u$ and $v$, then vertex $u$ (and arc $a$) is adjacent to vertex $v$, and vertex $v$ (and arc $a$) is adjacent from $v$. Let $\Gamma_G^+(v)$ and $\Gamma_G^-(v)$ denote the set of vertices adjacent from and to vertex $v$, respectively. Their cardinalities are the \emph{out-degree} $\delta_G^+(v)=|\Gamma_G^+(v)|$ of vertex $v$, and the \emph{in-degree} $\delta_G^-(v)=|\Gamma_G^-(v)|$ of vertex $v$. Digraph $G$ is called $d$\emph{-out-regular} if $\delta_G^+(v)=d$ for all $v\in V$, $d$\emph{-in-regular} if $\delta_G^-(v)=d$ for all $v\in V$, and $d$\emph{-regular} if $\delta_G^+(v)=\delta_G^-(v)=d$ for all $v\in V$.

\begin{figure}[t]
\vskip-.75cm
    \begin{center}
  \includegraphics[width=10cm]{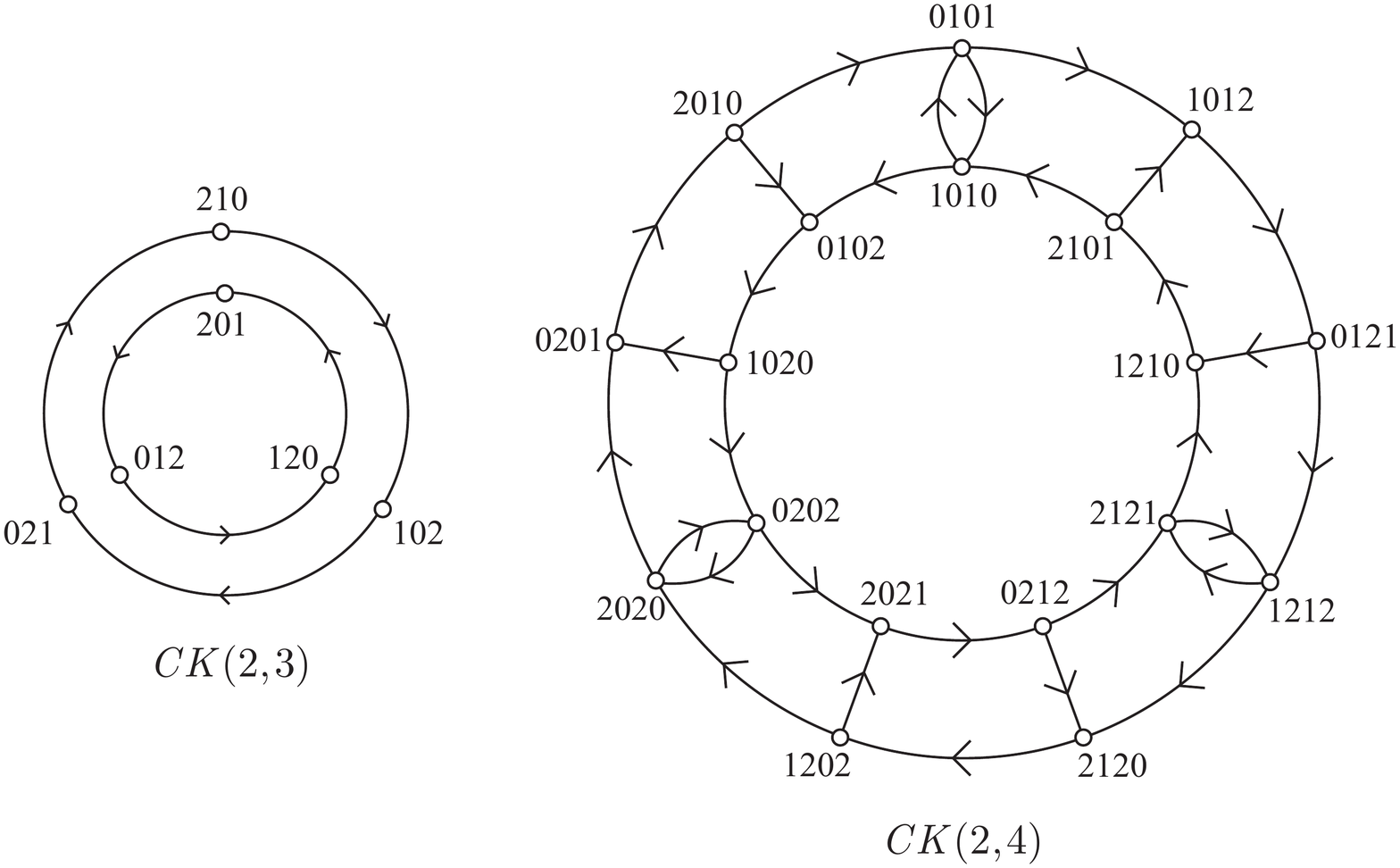}
  \end{center}
  \vskip-1.5cm
  \caption{The cyclic Kautz digraphs $CK(2,3)$ and $CK(2,4)$.}\label{fig:donut}
\end{figure}

\section{Parameters of the cyclic Kautz digraphs}
\label{sec:parametres}

\subsection{Numbers of vertices and arcs}

\begin{proposition}
\label{prop:num-vertexs}
The number of vertices of the cyclic Kautz digraph $CK(d,\ell)$ is $(-1)^\ell d+ d^\ell$.
\end{proposition}

\begin{proof}
There is a direct bijection between the vertices of $CK(d,\ell)$ and the closed walks of length~$\ell$ in the complete symmetric digraph with $d+1$ vertices (which is equivalent to the complete graph): A vertex of
$CK(d,\ell)$ is a sequence $a_1 a_2 \ldots a_\ell$ of different consecutive symbols, and with $a_1 \neq a_\ell$, from an alphabet of $d+1$ distinct symbols. Such a sequence corresponds to a closed walk $a_1 a_2 \ldots a_\ell a_1$ in the complete graph with $d+1$ vertices.
The claim now follows from the fact that the number of closed walks of length $\ell$ in a graph equals
the trace of $A^\ell$, where $A$ is the adjacency matrix of the graph (see, for example, Brouwer and Haemers~\cite{haemers}).
The spectrum of a complete graph with $d+1$ vertices has eigenvalue $-1$ with multiplicity $d$ and eigenvalue $d$ with multiplicity $1$.
Therefore, for this graph, the trace of $A^\ell$ is $(-1)^\ell d+ d^\ell.$
\end{proof}

\begin{proposition}
\label{prop:num-arcs}
The number of arcs of the cyclic Kautz digraph $CK(d,\ell)$ for $\ell\geq3$ is
$$
(d+1)d^{\ell}-(2d-1)((-1)^{\ell-1}d+d^{\ell-1}).
$$
\end{proposition}

We defer the proof to Section~\ref{sec:digrafs-linia}; the number of arcs of $CK(d,\ell)$ equals the number of vertices of the line digraph of $CK(d,\ell)$. Therefore, Proposition~\ref{prop:num-arcs} is the case $t=1$ of Theorem~\ref{theo:num-v-line-dig}. Note that for $\ell=2$, $CK(d,2)$ is the Kautz digraph $K(d,2)$ and its number of arcs is $(d+1)d^2$.


\subsection{Diameter}

\input{diameter.tex}

\section{The modified cyclic Kautz digraphs}
\label{sec:modificats}

\begin{figure}[t]
\vskip-.75cm
    \begin{center}
  \includegraphics[width=10cm]{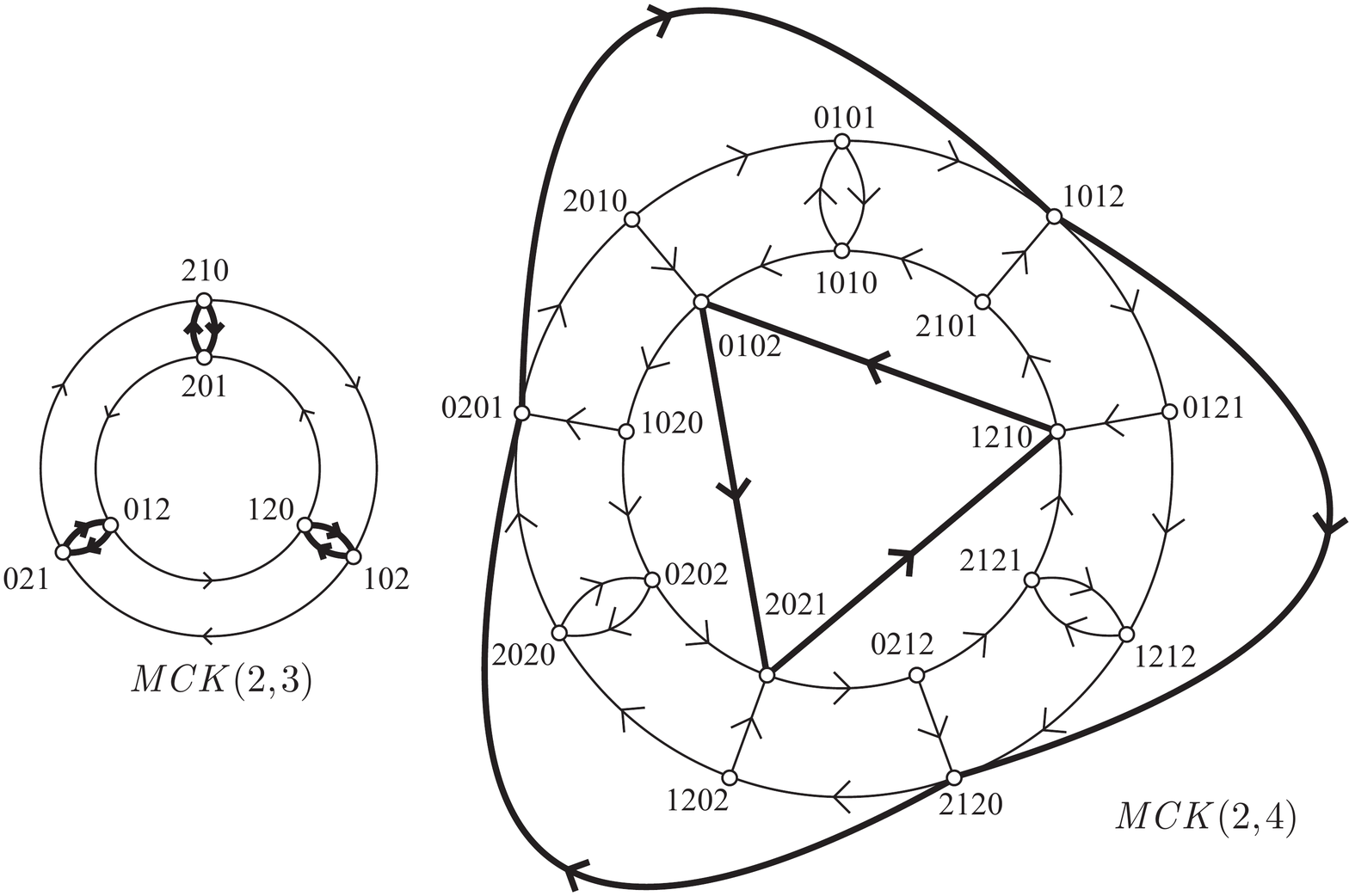}
  \end{center}
  \vskip-.5cm
  \caption{Modified cyclic Kautz digraphs $MCK(2,3)$ and $MCK(2,4)$ (the thick lines are the arcs added with respect to the corresponding cyclic Kautz digraphs).}\label{fig:donut-modificat}
\end{figure}

Recall that the diameter of the Kautz digraphs is optimal, that is, for a fixed out-degree $d$ and number of vertices $(d+1)d^{\ell-1}$, the Kautz digraph $K(d,\ell)$ has the smallest diameter $(D=\ell)$ among all digraphs with $(d+1) d^{\ell-1}$ vertices and degree $d$ (see  Li, Lu and Su~\cite{LiLuSu04}).
Since the diameter of the cyclic Kautz digraphs $CK(d,\ell)$ is greater than the diameter of the Kautz digraphs $K(d,\ell)$, we construct the \emph{modified cyclic Kautz digraphs} $MCK(d,\ell)$ by adding some arcs to $CK(d,\ell)$, in order to obtain the same diameter as $K(d,\ell)$.

In cyclic Kautz digraphs $CK(d,\ell)$, a vertex $a_2\ldots a_{\ell+1}$ is forbidden if $a_2=a_{\ell+1}$. For each such vertex, we replace the first symbol $a_2$ by one of the possible symbols $a_2'$ such that now $a_2'\neq a_3, a_{\ell+1}$ (so $a_2'\ldots a_{\ell+1}$ represents an allowed vertex). Then, we add arcs from vertex $a_1\ldots a_{\ell}$ to vertex $a_2'\ldots a_{\ell+1}$, with $a_1\neq a_{\ell}$ and $a_2'\neq a_3,a_{\ell+1}$. Note that $CK(d,\ell)$ and $MCK(d,\ell)$ have the same vertices, because we only add arcs to $CK(d,\ell)$ to obtain $MCK(d,\ell)$. See a pair of examples of modified cyclic Kautz digraphs in Figure~\ref{fig:donut-modificat}.

As said in the introduction, the Kautz digraphs $K(d,\ell)$ can be defined as iterated line digraphs of the complete symmetric digraphs $K_{d+1}$ (see Fiol, Yebra and Alegre~\cite{FiYeAl84}). This also means that the Kautz digraphs $K(d,\ell)$ can be obtained as the line digraph of $K(d,\ell-1)$. Namely,
\begin{eqnarray*}
  K(d,\ell) &=& L^{\ell-1}(K_{d+1}), \\
  K(d,\ell) &=& L(K(d,\ell-1)),
\end{eqnarray*}
where $L$ is the (1-iterated) line digraph of a digraph, and $L^{t}$ is the $t$-iterated line digraph.

Fiol and Llad\'{o}~\cite{FiLl92} defined the partial line digraph as follows:
Let $E'\subseteq E$ be a subset of arcs which are adjacent to all vertices of $G$, that is, $\{v;(u,v)\in E'\}=V$. A digraph $PL(G)$ is said to be a \emph{partial line digraph} of $G$ if its vertices
represent the arcs of $E'$, that is, $V(PL(G))=\{uv;(u,v)\in E'\}$, and a vertex $uv$ is adjacent to vertices $v'w$, for each $w\in\Gamma_G^+(v)$,
where
$$
v'=\left\{
\begin{array}{ll}
     v & \mbox{if } vw\in V(PL(G)), \label{cas1}\\
     \mbox{any other element of } \Gamma_G^-(w) \mbox{ such that } v'w\in V(PL(G))& \mbox{otherwise}\label{cas2}.
   \end{array}
\right.
$$
See an example of this definition in Figure~\ref{fig:dlp-Kautz}.

\begin{theorem}[\cite{FiLl92}]
\label{diam-PLG}
Let $G$ be a $d$-out-regular digraph $(d>1)$ with order $N$ and diameter $D$. Then, the order $N_{PL}$ and diameter $D_{PL}$ of a partial line digraph $PL(G)$ satisfy
\begin{eqnarray*}
  && N \leq N_{PL}\leq dN,\\
  && D \leq D_{PL}\leq D+1.
\end{eqnarray*}
Moreover, $PL(G)$ is also $d$-out-regular.
\end{theorem}

Note that, according to Fiol and Llad\'{o}~\cite{FiLl92}, $N_{PL}=N$ and $D_{PL}=D$ if and only if $PL(G)$ is $G$.

\begin{theorem}
The modified cyclic Kautz digraph $MCK(d,\ell)$ has the following properties:
\begin{itemize}
  \item[$(a)$] It is $d$-out-regular.
  \item[$(b)$] Its diameter is $D=\ell$, which is the same as the diameter of the Kautz digraph $K(d,\ell)$.
\end{itemize}
\end{theorem}

\begin{proof}
First, we show that the modified cyclic Kautz digraph $MCK(d,\ell)$ can be obtained as a partial line digraph of Kautz digraph $K(d,\ell-1)$:
$$
MCK(d,\ell)=PL(K(d,\ell-1)).
$$
See a scheme of the process of obtaining $MCK(d,\ell)$ from $K(d,\ell-1)$ in Figure~\ref{fig:dlp-Kautz}.
The vertices of $MCK(d,\ell)$ are the same vertices as the ones of $CK(d,\ell)$, that is, sequences $a_1\ldots a_{\ell}$ with $a_i\neq a_{i+1}$ and $a_1\neq a_\ell$, for $i=1,\ldots,\ell-1$. To obtain a $MCK(d,\ell)$, we add some arcs to $CK(d,\ell)$. As vertex $a_2a_3\ldots a_\ell a_2$ is forbidden, it does not belong to $MCK(d,\ell)$. The added arcs are the following:
$$
a_1a_2a_3\ldots a_\ell \longrightarrow a_2'a_3\ldots a_\ell a_2.
$$

Moreover, the vertices of $K(d,\ell-1)$ are $a_1\ldots a_{\ell-1}$ with $a_i\neq a_{i+1}$ for $i=1,\ldots,\ell-2$, and have arcs between vertices $a_1\ldots a_{\ell-1}$ and $a_2\ldots a_{\ell}$. Then, the arcs of $K(d,\ell-1)$ are $a_1\ldots a_{\ell}$, with $a_i\neq a_{i+1}$ for $i=1,\ldots,\ell-1$.
From these arcs, $E'\subseteq E$ is the subset of arcs that satisfies $a_1\neq a_\ell$. Now we apply the partial digraph technique to $E'$. Replacing the arcs of $E'$ by vertices, we obtain that the vertices of $PL(K(d,\ell-1))$ are $a_1\ldots a_{\ell}$, with $a_i\neq a_{i+1}$ and $a_1\neq a_{\ell}$, for $i=1,\ldots,\ell-1$. According to the definition of partial line digraph, there are two kinds of arcs in $PL(K(d,\ell-1))$. The first kind of arcs goes from vertex $a_1\ldots a_{\ell}$ to vertex $a_2\ldots a_{\ell+1}$, both vertices belonging to the set of vertices of $PL(K(d,\ell-1))$. The second kind of arcs goes from vertex $a_1\ldots a_{\ell}$ to vertex $a_2'a_3\ldots a_{\ell}a_2$, where we replaced the forbidden vertex $a_2a_3\ldots a_{\ell}a_2$ in $PL(K(d,\ell-1))$ for $a_2'a_3\ldots a_{\ell}a_2$, for a value of $a_2'$ such that $a_2'\neq a_2,a_3$.
As we obtain the same vertices and arcs in $MCK(d,\ell)$ as in $PL(K(d,\ell-1))$, they are the same digraph.

\begin{itemize}
  \item[$(a)$] As $K(d,\ell-1)$ is $d$-out-regular (indeed, it is $d$-regular), then by Theorem~\ref{diam-PLG} its partial line digraph $PL(K(d,\ell-1))=MCK(d,\ell)$ is also $d$-out-regular.
  \item[$(b)$] As the diameter of $K(d,\ell-1)$ is $D=\ell-1$, then by Theorem~\ref{diam-PLG} the diameter of its partial line digraph $PL(K(d,\ell-1))=MCK(d,\ell)$ is $\ell$. The diameter of $PL(K(d,\ell-1))$ cannot be $\ell-1$, because $PL(K(d,\ell-1))\neq K(d,\ell-1)$. Then, the diameter of $MCK(d,\ell)$ is $\ell$, which is the same as the diameter of $K(d,\ell)$.
\end{itemize}
\end{proof}


\begin{figure}[t]
    \begin{center}
  \includegraphics[width=14cm]{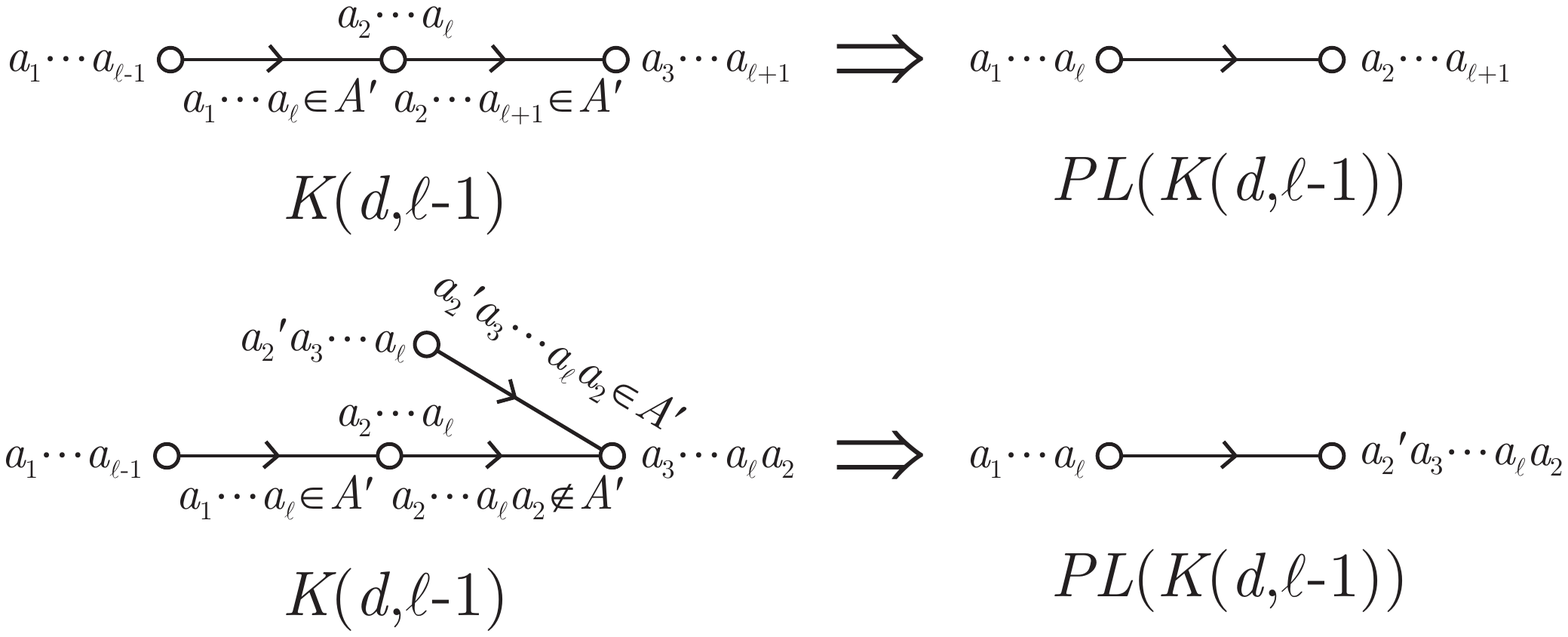}
  \end{center}
  \vskip-6cm
  \caption{Scheme of the two kinds of situation in the partial line digraph of $K(d,\ell-1)$.}\label{fig:dlp-Kautz}
\end{figure}

\section{Line digraphs iterations of the cyclic Kautz digraphs}
\label{sec:digrafs-linia}

As it was done with the Kautz digraphs $K(d,\ell)$, which are regular as said before, here we compute the number of vertices of the $t$-iterated line digraphs of the cyclic Kautz digraphs $CK(d,\ell)$, which are non-regular digraphs. In contrast with the regular digraphs, the resolution of the non-regular case is not immediate. 

As said in the Introduction, the diameter of a line digraph $L(G)$ of a digraph $G$ is $D(L(G))=D(G)+1$, even if $G$ is a non-regular digraph with the exception of directed cycles (see Fiol, Yebra and Alegre~\cite{FiYeAl84}). Then, with the line digraph technique, we obtain digraphs with minimal diameter and maximum connectivity from a certain iteration. For this reason, we calculate the number of vertices of digraphs obtained with this technique. In Figure~\ref{fig:CK(2,4)+L} there is an example of a $CK(d,\ell)$ and its line digraph. 

\input{kautz_recurrencia_digraf_linia_13_09_2014.tex}

Theorem~\ref{theo:num-v-line-dig} gives the number of vertices at the $t$-iteration of the line digraph of $CK(d,\ell)$, with $1\leq t\leq \ell-2$. Now we compute the number of vertices of $L^t(CK(d,\ell))$ without restriction on the value of $t$, for the particular case $\ell=4$. Let $N_t=|L^{t}(CK(d,4))|$.

\begin{proposition}
\label{prop:v-lin-dig2}
The number of vertices $N_t$ of the iterated line digraph of $CK(d,4)$ at iteration $t\geq0$ is
  $$
  N_t=\alpha\left(\frac{d-1+\sqrt{d^2-2d+5}}{2}\right)^t+\beta\left(\frac{d-1-\sqrt{d^2-2d+5}}{2}\right)^t,
  $$
  where $\alpha=\frac{1}{2}d(d+1)\left(d^2-d+1+\frac{d^3-2d^2+4d-1}{\sqrt{d^2-2d+5}}\right)$ and $\beta=\frac{1}{2}d(d+1)\left(d^2-d+1-\frac{d^3-2d^2+4d-1}{\sqrt{d^2-2d+5}}\right)$.
Moreover, if $d=2$, $N_t$ follows a Fibonacci sequence with initial values $18$ and $30$.
\end{proposition}

\begin{figure}[t]
\vskip-1cm
    \begin{center}
  \includegraphics[width=10cm]{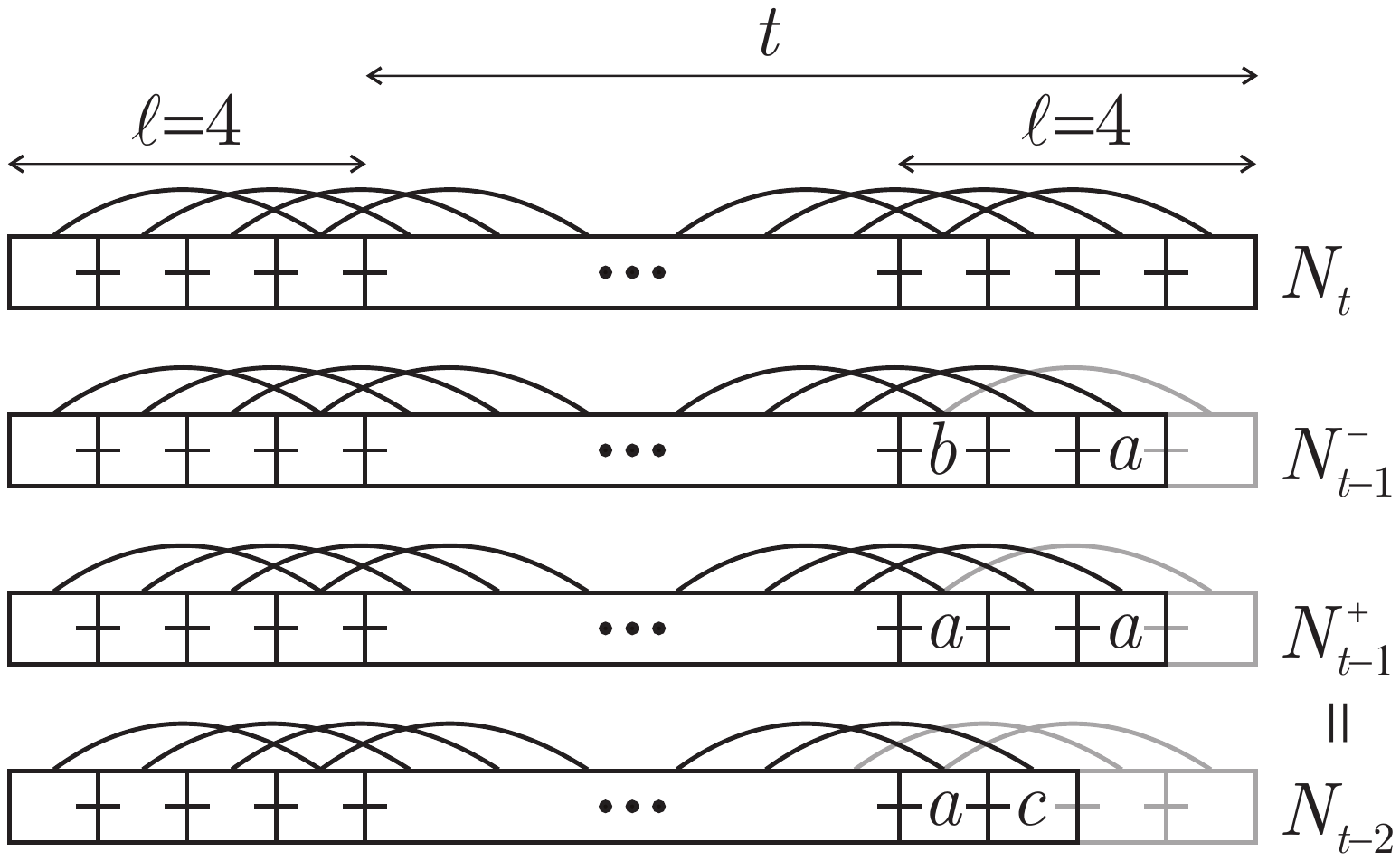}
  \end{center}
  \vskip-9.5cm
  \caption{Scheme to obtain the number of vertices in the $t$-iterated line digraph of $CK(d,4)$. We indicate in grey the characters that have to be added to obtain a sequence with $t+4$ characters. The lines between two symbols represent that they must be different.}
  \label{fig:num-v-dig-lin}
\end{figure}

\begin{proof}
  In the following, we construct $L^t(CK(d,\ell))$ by adding a new symbol $a_{4+t}$ to every vertex $a_1\ldots a_{3+t}$ of $L^{t-1}(CK(d,\ell))$.
   We make a partition of $L^{t-1}(CK(d,4))$ into two sets, one with sequences such that $a_{t+1}=a_{t+3}$, and the other one with sequences that satisfy $a_{t+1}\neq a_{t+3}$. In the first case, the cardinality of the number of vertices is called $N_{t-1}^{+}$, and in the second one it is called $N_{t-1}^{-}$, satisfying $N_{t-1}=N^{+}_{t-1}+N^{-}_{t-1}$. See a scheme for $\ell=4$ in Figure~\ref{fig:num-v-dig-lin}, where 
   the lines drawn between symbols
   indicate that the symbols at these positions must be different.

  With $d+1$ symbols, given the characters $a_{t+1}$ and $a_{t+3}$, for $a_{t+4}$ there are $d$ possible symbols corresponding to $N^{+}_{t-1}$ (that is, with $a_{t+1}=a_{t+3}$), and $d-1$ corresponding to $N^{-}_{t-1}$ (that is, with $a_{t+1}\neq a_{t+3}$). Then, $N_t=d\,N^{+}_{t-1}+(d-1)N^{-}_{t-1}$. Moreover, $N_{t-2}=N^{+}_{t-1}$, because from a vertex of $L^{t-2}(CK(d,4))$ there is only one possible symbol for the character $a_{t+3}$ to obtain a vertex of $L^{t-1}(CK(d,4))$, such that $a_{t+1}=a_{t+3}$. See Figure~\ref{fig:num-v-dig-lin}. Thus, $N^{-}_{t-1}=N_{t-1}-N_{t-2}$, and $N_t$ satisfies the recurrence equation $N_t=(d-1)N_{t-1}+N_{t-2}$, with initial conditions $N_0=d^4+d$ (see Proposition~\ref{prop:num-vertexs}) and $N_1=d(d+1)(d^3-2d^2+3d-1)$ (see Proposition~\ref{prop:num-arcs}).
  Solving this recurrence equation, we have
      $$
      N_t=\alpha\left(\frac{d-1+\sqrt{d^2-2d+5}}{2}\right)^t+\beta\left(\frac{d-1-\sqrt{d^2-2d+5}}{2}\right)^t.
      $$
     With the initial conditions $N_0$ and $N_1$, we get
      \begin{eqnarray*}
        \alpha &=& \frac{1}{2}d(d+1)\left(d^2-d+1+\frac{d^3-2d^2+4d-1}{\sqrt{d^2-2d+5}}\right), \\
        \beta &=& \frac{1}{2}d(d+1)\left(d^2-d+1-\frac{d^3-2d^2+4d-1}{\sqrt{d^2-2d+5}}\right).
      \end{eqnarray*}
      As a particular case, if $d=2$, then the recurrence equation is $N_t=N_{t-1}+N_{t-2}$, with initial values $N_0=18$ and $N_1=30$. Then, $N_t=\left(9+\frac{21}{\sqrt{5}}\right)\left(\frac{1+\sqrt{5}}{2}\right)^t+\left(9-\frac{21}{\sqrt{5}}\right)\left(\frac{1-\sqrt{5}}{2}\right)^t$,
      and we obtain the values $N_t=18,30,48,78,\ldots$ for $t=0,1,2,3,\ldots$, which is a Fibonacci sequence with initial values $18$ and $30$.
\end{proof}


We leave as open problems the following two questions: Find the number of vertices of the $t$-iterated cyclic Kautz digraph $CK(d,\ell)$
for the remaining values of $d$, $t$ and $\ell$; and compute
the number of vertices of the $t$-iterated modified cyclic Kautz digraph $MCK(d,\ell)$.

\newpage
\section*{Appendix}

\input{appendix.tex}

\end{document}

%% file: diameter.tex
\newcommand{\comment}[1]{\em \footnotesize #1}

In this section we give a complete characterization of the diameter of
the cyclic Kautz digraphs $CK(d,\ell)$, depending on the values of $d$
and $\ell$.
%
%
In the following claims and proofs, let us fix $\Sigma = \{0,1,{\dots},d\}$
to be the alphabet of $CK(d,\ell)$.

First, we give the diameter for small values of $d$ and $\ell$.
Some proofs are deferred to the Appendix.

\begin{lemma}[$\ell=1$) \& ($d = 1,\ell \geq 2$]
\label{lem:diamd1}
The diameter of $CK(d,1)$ is $1$.
The cyclic Kautz digraphs $CK(1,\ell \geq 2)$ exist only if $\ell$ is
even.
For even $\ell$, the diameter of $CK(1,\ell \geq 2)$ is $1$.
\end{lemma}


\begin{lemma}[$d \geq 2,\ell = 2$]
\label{lem:diamd2}
The diameter of $CK(d \geq 2,2)$ is $2$.
\end{lemma}


Before we proceed to discuss the diameter of $CK(d,\ell)$ for the
remaining cases of $d$ and $\ell$, and as a tool for the proofs, let us introduce
the concept of the \emph{disc representation}.
\input{disc_representation.tex}
\input{imprint.tex}

We now use Lemma~\ref{lem:pathIfImprint} to prove the following
result.

\begin{lemma}[$d = 2,\ell \geq 3$]
\label{lem:diam2var}
The diameter of $CK(2,4)$ is $7$, and the diameter of $CK(2,3)$ and of
$CK(2,\ell \geq 5)$ is infinite.
\end{lemma}

\begin{proof}
We could argue that the diameter of $CK(2,4)$ is finite by showing
that all the vertices of $CK(2,4)$ have imprint that contains exactly
two $+$ signs and two $-$ signs.
In fact, the diameter of $CK(2,4)$ is $7$, which can be checked in
Figure~\ref{fig:donut}.

\begin{table}[t]
\caption{
Two vertices $u$ and $v$ of $CK(2,\ell)$,
for $\ell = 3$  and for $\ell \geq 5$,
such that the numbers of $+$ and $-$ signs in their imprints
differ, implying that $u$ and $v$ belong to different connected
components.}
\center{ \begin{tabular}{l|c|c}
 & $u$, imprint of $u$ & $v$, imprint of $v$ \\
\hline
$\ell=3$ & ${0 {\bf 1 2}}$ & ${0 {\bf 2 1}}$ \\
 & $(+  +  +)$ & $(-  -  -)$ \\
\hline
$\ell=3r, \ell \geq 5$ &
 ${0 {\bf 1 2} 0 1 2 {\ldots} 0 1 2}$ &
 ${0 {\bf 2 1} 0 1 2 {\ldots} 0 1 2}$ \\
 &
 $(+  +  +  +  +  +  {\ldots}  +  +  + )$ &
 $(-  -  -  +  +  +  {\ldots}  +  +  +)$ \\
\hline
$\ell=3r+1, \ell \geq 5$ &
 ${0 {\bf 1 2} 0 1 2 {\ldots} 0 1 2 1}$ &
 ${0 {\bf 2 1} 0 1 2 {\ldots} 0 1 2 1}$ \\
 &
 $(+  +  +  +  +  +  {\ldots}  +  +  -  -)$ &
 $(-  -  -  +  +  +  {\ldots}  +  +  -  -)$ \\
\hline
$\ell=3r+2, \ell \geq 5$ &
 ${0 {\bf 1 2} 0 1 2 {\ldots} 0 1 2 0 1}$ &
 ${0 {\bf 2 1} 0 1 2 {\ldots} 0 1 2 0 1}$ \\
 &
 $(+  +  +  +  +  +  {\ldots}  +  +  +  +  -)$ &
 $(-  -  -  +  +  +  {\ldots}  +  +  +  +  -)$ \\
\end{tabular} }
\label{tab-not-connected-vertices}
\end{table}
For $\ell = 3$  and for $\ell \geq 5$ we show that the diameter
of $CK(2,\ell)$ is infinite.
In particular, we distinguish four cases depending on the value of
$\ell$ and, for each case, we pick two vertices $u$ and $v$ of
$CK(2,\ell)$ such that the numbers of $+$ and $-$ signs in the imprints
of the two vertices differ (see Table~\ref{tab-not-connected-vertices}).
Then, by Lemma~\ref{lem:pathIfImprint}, there is no path from $u$ to
$v$ in $CK(2,\ell)$.
\end{proof}

We proceed with the remaining values of $d$ and $\ell$, and give
matching upper and lower bounds on the diameter of $CK(d,\ell)$.

\begin{lemma}[$d=3, \ell \geq 3$, upper bound]
\label{lem:diam41}
The diameter of $CK(3,\ell \geq 3)$ is at most $2\ell - 1$.
\end{lemma}

\begin{proof}
We prove the claim by showing that between any two vertices $u$, $v$ of
$CK(3,\ell \geq 3)$, there is always a path of length at most $2\ell
- 1$ or $2\ell - 2$.
In particular, we will show that there is either a sequence
$x = {x_1  x_2  {\ldots}  x_{\ell-1}}$ of $\ell - 1$ symbols or
a sequence $y = {y_1  y_2  {\ldots}  y_{\ell-2}}$ of $\ell - 2$ symbols
such that by concatenating the sequences of $u$, one of $x$ or $y$,
and $v$, we obtain a sequence with the property that any
contiguous subsequence of length $\ell$ forms a vertex of
$CK(3,\ell \geq 3)$.

Given the vertices $u$ and $v$, let us first try to find such sequence
$x$ of $\ell - 1$ symbols.
From the definition of a vertex of a cyclic Kautz digraph it follows that
in the desired sequence, the symbol $x_i$ must differ from the symbols
$x_{i-1}$, $x_{i+1}$, $u_{i+1}$, and $v_i$ (also, $x_1$ differs from
$u_\ell$, and $x_{\ell-1}$ differs from $v_1$).
We adopt the following strategy. We choose the symbols for $x_1, x_2,
{\dots}, x_{\ell-1}$ one by one in this order. Our alphabet has $4$
symbols, but for each $x_i$ (with the exception of the last one) there
are exactly $3$ restrictions given. We therefore can choose the
symbols for $x_1, x_2, {\dots}, x_{\ell-2}$ that meet the
requirements.
However, a problem may arise when choosing the symbol for
$x_{\ell-1}$.
For $x_{\ell-1}$ we need to choose a symbol that differs from
$u_\ell$, $x_{\ell-2}$, $v_1$, and $v_{\ell-1}$.
We distinguish two cases.
If any two of these $4$ symbols are the same, there remains a symbol
which can be assigned to $x_{\ell-1}$ and the constructed sequence $x$
of length $\ell - 1$ satisfies the requirements.
Otherwise, if all the $4$ symbols $u_\ell$, $x_{\ell-2}$, $v_1$, and
$v_{\ell-1}$ differ, say
$u_\ell = a$, $x_{\ell-2} = b$, $v_1 = c$, and $v_{\ell-1} = d$,
we cannot assign a symbol to $x_{\ell-1}$ that would satisfy the
requirements.
We have the following observations:
\begin{itemize}
\item[1)]
$u_\ell \not= v_1$;
\item[2)]
$u_\ell = a$ implies $u_{\ell-1} \not= a$;
\item[3)]
$x_{\ell-2} = b$ implies $u_{\ell-1} \not= b$.
\end{itemize}
We interrupt the search for a sequence $x$ of $\ell - 1$ symbols.
Instead, we find a sequence $y$ of $\ell - 2$ symbols as follows.
In the desired sequence, a symbol $y_i$ must differ from the
symbols $y_{i-1}$, $y_{i+1}$, $u_{i+1}$, and $v_{i+1}$ (also, $y_1$
differs from $u_\ell$, and $y_{\ell-2}$ differs from $v_1$).
Moreover, since $y$ has length $\ell-2$, the subsequence
${u_\ell y_1 {\ldots} y_{\ell-2} v_1}$ also needs to be a vertex of
$CK(3,\ell \geq 3)$, thus $u_\ell$ must differ from $v_1$.
The condition $u_\ell \not= v_1$ is satisfied by the above observation
1). To choose the symbols for $y_1, y_2, {\dots}, y_{\ell-2}$ we adopt
a similar strategy as above. Again, we can choose the symbols for
$y_1, y_2, {\dots}, y_{\ell-3}$ one by one in this order and meet the
requirements. Finally, for $y_{\ell-2}$ we need to choose a symbol
that differs from $u_{\ell-1}$, $y_{\ell-3}$, $v_1$, and $v_{\ell-1}$.
Since the previous search for $x$ failed, we know that $v_1 = c$,
$v_{\ell-1} = d$, and $u_{\ell-1} \not= a,b$. This implies that
$u_{\ell-1}$ has the same symbol as either $v_1$ or $v_{\ell-1}$, and
thus there remains one symbol which can be assigned to $y_{\ell-2}$.

Therefore, we can either find a valid sequence $x$ of length $\ell-1$,
or a valid sequence $y$ of length $\ell-2$. In both cases this gives
us an upper bound of $2\ell-1$ on the length of the  shortest path
between any pair of nodes of $CK(3,\ell \geq 3)$.
\end{proof}

\begin{lemma}[$d \geq 4, \ell \geq 4$, upper bound]
\label{lem:diam51}
The diameter of $CK(d\geq 4,\ell \geq 4)$ is at most $2\ell - 2$.
\end{lemma}

\begin{proof}
We use the same strategy as in Lemma~\ref{lem:diam41}.
We show that for any pair of vertices $u$, $v$ of $CK(d\geq 4,\ell
\geq 4)$, there
is always a path of length at most $2\ell - 2$ or $2\ell - 3$.
In particular, we show that there is either a sequence
$y = {y_1  y_2  {\ldots}  y_{\ell-2}}$ of $\ell-2$ symbols or a sequence
$z = {z_1  z_2  {\ldots}  z_{\ell-3}}$ of $\ell-3$ symbols, such that the
concatenation of $u$, one of $y$ or $z$, and $v$ forms
a sequence such that any contiguous subsequence of length $\ell$
is a vertex of $CK(d\geq 4,\ell \geq 4)$.
The details are given in the Appendix.
\end{proof}

\begin{lemma}[($d\geq 3, \ell = 3$) \& ($d=3, \ell \geq 3$) \& ($d \geq 4, \ell \geq 3$), lower bounds]
$(a)$ The diameter of $CK(d \geq 3,3)$ is at least $5 = 2\ell - 1$.
$(b)$ The diameter of $CK(3,\ell \geq 3)$ is at least $2\ell - 1$.
$(c)$ The diameter of $CK(d\geq 4,\ell \geq 3)$ is at least $2\ell - 2$.
\label{lem:diam_lb}
\end{lemma}

\begin{proof}
To prove a lower bound on the diameter of $CK(d,\ell)$, it suffices to
identify two vertices that are at the claimed distance.
In particular, we can argue the following (for more details, see
the proof of this lemma in the Appendix).
\begin{itemize}
\item The vertices $u = {0 1 2}$ and $v = {2 1 0}$ of $CK(d \geq 3,3)$
are at the distance $5 = 2\ell - 1$.
\item
If $\ell$ is odd, $u = {0 1 0 1 {\ldots} 0 1 2}$ and $v
= {2 1 0 {\ldots} 1 0 1 0}$ of $CK(3,\ell \geq 3)$ are at the distance
$2\ell - 1$.
\item
If $\ell$ is even, the vertices $u = {1 0 2 0 2 0 {\ldots}
2 0 1 2}$ and $v = {2 1 3 0 2 0 {\ldots} 2 0 1 0}$ of $CK(3,\ell
> 3)$ (for example, for $\ell=4$, $u = {2 0 1 2}$, and
$v = {2 1 3 0}$) are at the distance $2\ell - 1$.
\item
Concerning $CK(d\geq 4,\ell \geq 3)$, the vertices
$u = {{\ldots}  0 1 0 1 0 1 2}$ ($u$ begins with ${0 1}$ if
$\ell$ is odd and with ${1 0}$ if $\ell$ is even),
and $v = {1 3 2 0 2 0 2 {\ldots}}$
of $CK(d\geq 4,\ell \geq 3)$ are at the distance
$2\ell - 2$.
\end{itemize}
\end{proof}

\begin{figure}[t]
   \centering{\includegraphics[scale=.5]{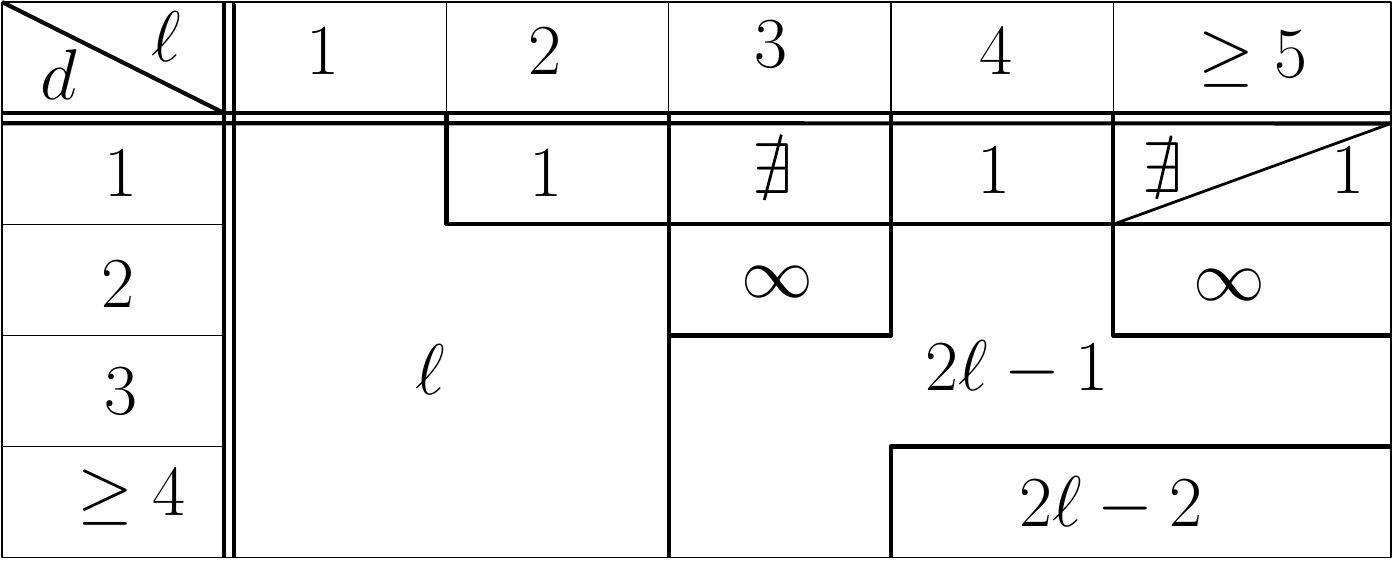}}
\caption{ Summary of the diameter of $CK(d,\ell)$, depending on the
values of $d$ and $\ell$. }
\label{fig-diam-overview}
\end{figure}
The above lemmas specify the diameter of $CK(d,\ell)$ for all the
values of $d$ and $\ell$; we summarize them into the following theorem
(see Figure~\ref{fig-diam-overview} for an overview).

\begin{theorem}
The diameter of $CK(d,1)$ is $1$.
The diameter of $CK(1,\ell \geq 2)$ is $1$ if $\ell$ is even ($CK(1,\ell \geq 2)$ does not exist if $\ell$ is odd).
The diameter of $CK(d \geq 2,2)$ is $2$.
The diameter of $CK(2,3)$ is infinite.
The diameter of $CK(2,4)$ is $7\ (= 2\ell - 1)$.
The diameter of $CK(2,\ell \geq 5)$ is infinite.
The diameter of $CK(d\geq 3,3)$ is $5\ (= 2\ell - 1)$.
The diameter of $CK(3,\ell \geq 3)$ is $2\ell - 1$.
Finally, the diameter of $CK(d\geq 4,\ell \geq 4)$ is $2\ell - 2$.
\end{theorem}

\begin{proof}
Follows directly from Lemmas
\ref{lem:diamd1}, 
\ref{lem:diamd2},
and~\ref{lem:diam2var}--\ref{lem:diam_lb}.
\end{proof}

%% file: disc_representation.tex
Each vertex $v$ of the cyclic Kautz digraph $CK(d,\ell)$ can be uniquely
represented on a disc with a marked start (see
Figure~\ref{fig-disc-rep}). We will refer to this as a \emph{disc
representation} of the vertex $v$ (in short, the {\em disc of} $v$).
In fact, there is a straightforward bijection between the vertices of
$CK(d,\ell)$ and the set of discs $D(d,\ell)$ with a marked start,
containing $\ell$ symbols of the alphabet of size $d+1$ in such a way
that no two consecutive symbols are the same.
Clearly, the set of discs $D(d,\ell)$ is closed under the operation of
\emph{swapping a symbol} in a disc in a valid way (that is, exchanging
the symbol for another symbol of the alphabet so that it is different
from both its neighbors). Also, the set $D(d,\ell)$ is closed under
the operation of \emph{moving the marked position} in a disc, which we
will also refer to as \emph{rotating the disc}.
In other words, by performing any of these operations on a disc
representation of a vertex of $CK(d,\ell)$, we again obtain a disc
representation of a vertex of $CK(d,\ell)$.

Observe that there is an arc from a vertex $u$ to a vertex $v$ in
$CK(d,\ell)$ if and only if the disc representation of $v$ can be
obtained from the disc representation of $u$, by swapping the symbol in
the marked position and moving the marked position one step clockwise
(which is equivalent to rotating the disc counterclockwise and leaving
the marked position at a fixed place, say at 12 hours).
We can prove the following (see the Appendix for the details).

\begin{figure}[t]
\centerline{
   \includegraphics[scale=.5]{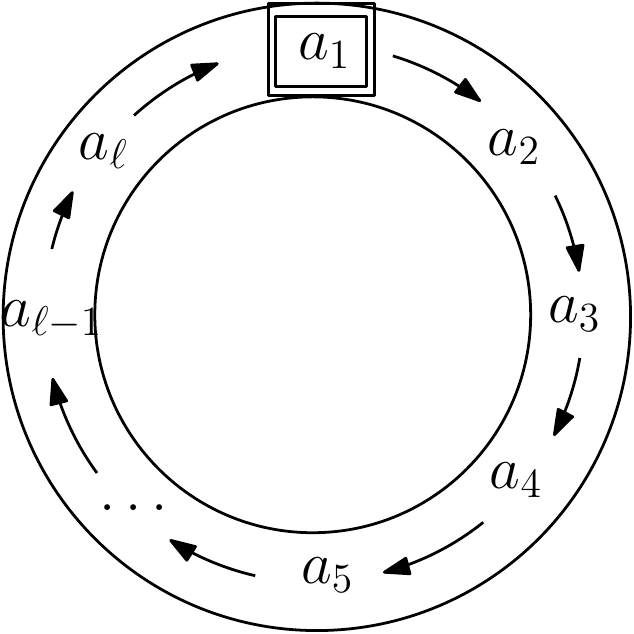}
   \qquad
   \qquad
   \includegraphics[scale=.5]{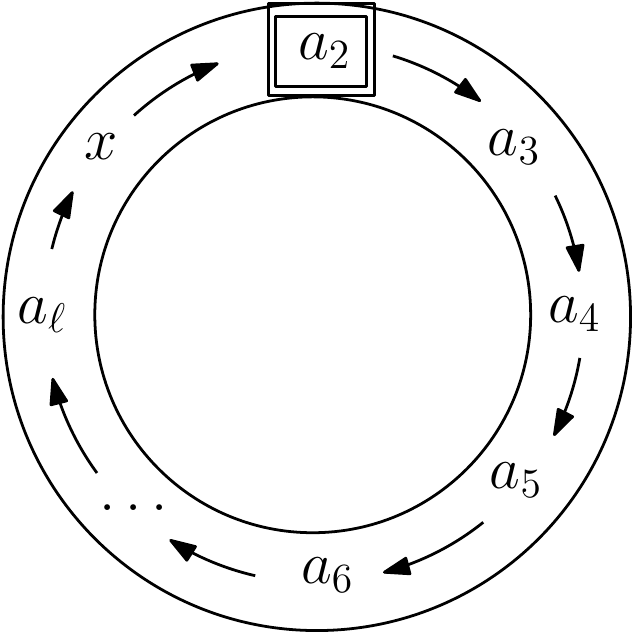}}
\caption{The disc representation of a vertex
$u = {a_1 a_2 {\ldots} a_{\ell-1} a_\ell}$ of $CK(d,\ell)$ and its
neighbor $v = {a_2 a_3 {\ldots} a_\ell x}$, that is, $(u,v)$
is an arc in $CK(d,\ell)$.  }
\label{fig-disc-rep}
\end{figure}

\begin{lemma}
\label{lem:pathIfOperations}
There is a path from a vertex $u$ to a vertex $v$ in $CK(d,\ell)$ if
and only if the disc of $v$ can be obtained from the disc of $u$
by a sequence of operations:
\begin{itemize}
\item Rotation of the disc,
\item Swap of one symbol (in a valid way).
\end{itemize}
\end{lemma}

%% file: imprint.tex
Let us now consider cyclic Kautz digraphs $CK(2,\ell)$ with alphabet $\Sigma
= \{0,1,2\}$,
and let us define a function $sgn: \Sigma^2 \rightarrow \{+,-\}$, which
assigns a $+/-$ sign to an ordered pair of distinct symbols $(a,b)$ as
follows.
\begin{itemize}
\item[] $sgn(0,1) = sgn(1,2) = sgn(2,0) = +$,
\item[] $sgn(1,0) = sgn(2,1) = sgn(0,2) = -$.
\end{itemize}
Given a vertex $v$ of $CK(2,\ell)$, we define an \emph{imprint} of
$v$ as a sequence $im(v)$ of length $\ell$ containing symbols $+$ and
$-$ as follows.
On the $i$-th position ($i \in \{1,\dots,\ell-1\}$)
of the sequence $im(v)$ there is the symbol $sgn(v_i,v_{i+1})$.
On the last position of $im(v)$ there is the symbol $sgn(v_\ell,v_{1})$.
Since each two consecutive symbols of $v$ are distinct, the imprint of
$v$ is well defined and unique.

The imprints have the following property (the proof is postponed to the 
Appendix).

\begin{lemma}
\label{lem:OpsNotChangeImprint}
Let $u$, $v$ be two vertices of $CK(2,\ell)$.
The imprints of $u$ and $v$ have the same number of $+$ and $-$ signs
if and only if the disc of $v$ can be obtained from the disc of $u$
by a sequence of operations:
\begin{itemize}
\item Rotation of the disc,
\item Swap of one symbol (in a valid way).
\end{itemize}
\end{lemma}

\begin{lemma}
\label{lem:pathIfImprint}
In $CK(2,\ell)$, there is a path from a vertex $u$ to a vertex $v$ if
and only if the imprint of $v$ contains the same number of $+$ and $-$
signs as the imprint of $u$.
\end{lemma}

\begin{proof}
Follows directly from
Lemmas~\ref{lem:pathIfOperations} and~\ref{lem:OpsNotChangeImprint}.
\end{proof}

%% file: kautz_recurrencia_digraf_linia_13_09_2014.tex
%
%



\begin{figure}[t]
\vskip-.25cm
    \begin{center}
  \includegraphics[width=10cm]{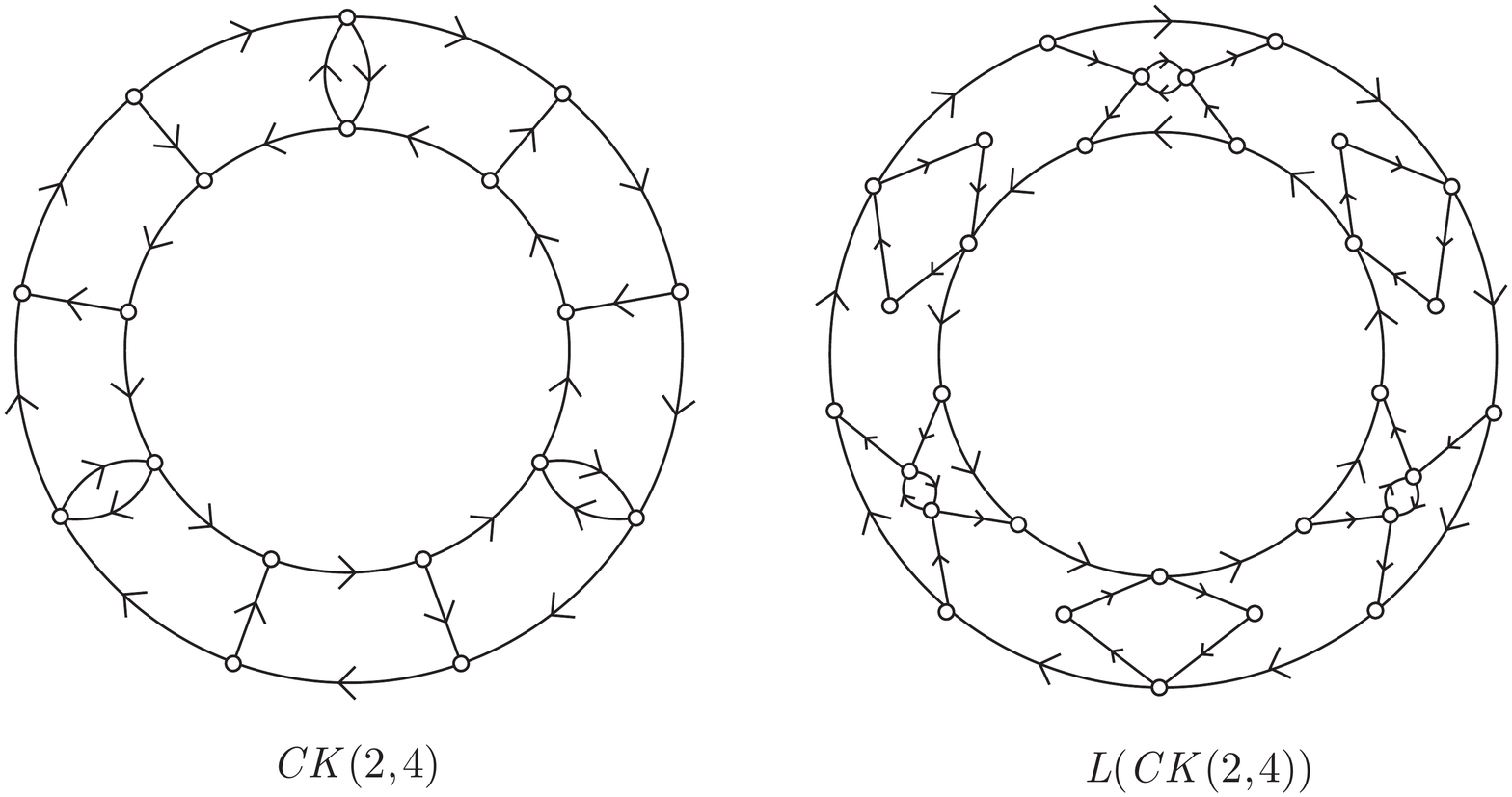}
  \end{center}
  \vskip-2.25cm
  \caption{$CK(2,4)$ and its line digraph at iteration $t=1$.}\label{fig:CK(2,4)+L}
\end{figure}

\begin{theorem}
\label{theo:num-v-line-dig}
Let $\ell\geq 3$ and $d\geq 1$ be integers. Then the number of vertices of the $t$-iterated line digraph $L^t(CK(d,\ell))$ of $CK(d,\ell)$, for $1\leq t\leq \ell-2$, is
$$
(d^2-d+1)^t d^{\ell-t}+\frac{1}{2}(-1)^{\ell+1}(d-2)^t(d-1)d +\frac{1}{2}(-1)^{\ell} d^{t+1}(d+1).
$$
\end{theorem}

Recall that a vertex of $CK(d,\ell)$ is a sequence  of $\ell$ characters from an alphabet of $d+1$ symbols, such that consecutive symbols, and also the first and last symbol, are different. Two vertices of $CK(d,\ell)$ are adjacent when they have the form  $a_1 a_2\ldots a_\ell$ and $a_2 \ldots a_\ell a_{\ell+1}$  (with $a_1\neq a_\ell$ and $a_2\neq a_{\ell+1}$). This suggests to represent an arc of $CK(d,\ell)$ as a sequence of $\ell+1$ characters  $a_1 a_2\ldots a_\ell a_{\ell+1}$ satisfying $a_i\neq a_{i+1}$ for $1\leq i \leq \ell$,  $a_1\neq a_\ell$, and $a_2\neq a_{\ell+1}$. Note that $a_1$ can be equal to $a_{\ell+1}.$ See Figure~\ref{fig:verticesIteradedLineGraph} for an example. The arcs of $CK(d,\ell)$ are the vertices of the iterated line digraph of $CK(d,\ell)$ at the first iteration $t=1.$ Two such vertices are adjacent when they have the form $a_1 a_2\ldots a_\ell a_{\ell+1}$ and $a_2 a_3\ldots a_{\ell+1} a_{\ell+2}$, with  $a_1\neq a_\ell$, $a_2\neq a_{\ell+1}$ and $a_3\neq a_{\ell+2}$. Therefore, a vertex of the iterated line digraph of a cyclic Kautz digraph $CK(d,\ell)$ at iteration $t=2$ can be represented by a sequence of $\ell+2$ characters satisfying $a_i\neq a_{i+1}$, $a_1\neq a_\ell$, $a_2\neq a_{\ell+1}$ and $a_3\neq a_{\ell+2}$.  In general, for $0\leq t\leq \ell-2$, the vertices of the iterated line digraph of $CK(d,\ell)$ at iteration $t$
are represented by sequences $a_1 a_2\ldots a_{\ell+t}$ satisfying $a_i \neq a_{i+1}$ for $1\leq i \leq \ell+t-1$ and
$a_i \neq a_{i+\ell-1}$ for $1\leq i\leq t+1$. We denote the set of these sequences by $\mathcal{S}$ for any $t$ and $\ell$. See Figure~\ref{fig:sequencesS}.

For the case $d=1$, $CK(d,\ell)$ has two vertices if $\ell$ is even, namely the two sequences of alternating characters, and  $CK(d,\ell)$ has no vertices if $\ell$ is odd. It is easy to verify that this also holds for the set $\mathcal{S}$.

Let us now prove the case $d\geq 2$.
To count the number of elements of $\mathcal{S}$, we first only count sequences of the form $a_1 a_2\ldots a_{\ell}$ for $\ell$ even, with $a_i \neq a_{i+1}$ for $1\leq i \leq \ell-1$ and $a_i \neq a_{i+\ell/2}$ for $1\leq i \leq \ell/2$. We denote this set of sequences of even length by $\mathcal{S'} \subset \mathcal{S}$.

We partition $\mathcal{S'}$ into two classes $\mathcal{C}_\ell$ and $\mathcal{D}_\ell$, where $\mathcal{C}_\ell$ is the set of those sequences of $\mathcal{S'}$ that have $a_{\ell/2+1}=a_\ell$, and $\mathcal{D}_\ell$ is the set of the remaining ones. We also introduce an auxiliary class of sequences $\mathcal{B}_\ell$, which is defined as $\mathcal{D}_\ell$ with the further restriction that $a_{\ell/2+1}=a_{\ell/2}$; hence, the elements of $\mathcal{B}_\ell$ are not sequences of $\mathcal{S}$. See Figure~\ref{fig:classesBCD}.
We denote the cardinalities of $\mathcal{C}_\ell, \mathcal{D}_\ell$ and $\mathcal{B}_\ell$ with $C_\ell, D_\ell$ and $B_\ell$, respectively.

For the first values of $\ell$, it is easy to calculate $B_4=(d+1)d^2$, $C_4=0$, and $D_4=(d+1)d(d-1)^2$ and
$$B_6=(d+1)d(d-1)^3$$
$$C_6=(d+1)d(d^3-2d^2+3d-1)$$
$$D_6=(d+1)d(d-1)^2(d^2-2d+3).$$
For $\ell>6$, we generate all sequences of $\mathcal{B}_\ell$, $\mathcal{C}_\ell$, and $\mathcal{D}_\ell$ from all sequences of $\mathcal{B}_{\ell-2}$, $\mathcal{C}_{\ell-2}$, and $\mathcal{D}_{\ell-2}.$ This is done by inserting a new character between $a_{\ell/2}$ and $a_{\ell/2+1}$, and another new character after $a_\ell.$ See Figure~\ref{fig:insertBCD}.
Let us now describe how to generate sequences from the class $\mathcal{B}_\ell$ using only the classes $\mathcal{C}_{\ell-2}$ and $\mathcal{D}_{\ell-2}$. There is only one possibility to insert a new character between $a_{\ell/2}$ and $a_{\ell/2+1}$, that is, $a_{new}=a_{\ell/2+1}.$ The other new character $a_{\ell+2}$ has to be different from $a_\ell$ and from $a_{new}=a_{\ell/2+1}.$ If we start with a sequence from $\mathcal{D}_{\ell-2}$, then $a_{new}\neq a_\ell$, and there are $d-1$ possible ways to insert character $a_{\ell+2}.$ If we start with a sequence from $\mathcal{C}_{\ell-2}$, then $a_{new} = a_\ell$, and there are $d$ possible ways to insert character $a_{\ell+1}.$ We therefore obtain
$$B_\ell=(d-1)D_{\ell-2}+d\,C_{\ell-2}.$$
Similar arguments show
$$C_\ell=(d-1)D_{\ell-2}+d\,B_{\ell-2},$$
$$D_\ell=(d^2-3d+3)D_{\ell-2}+(d-1)^2C_{\ell-2}+(d-1)^2B_{\ell-2}.$$
Note that every sequence is generated exactly once.

\begin{figure}[t]
  \begin{center}
  \includegraphics[scale=0.8]{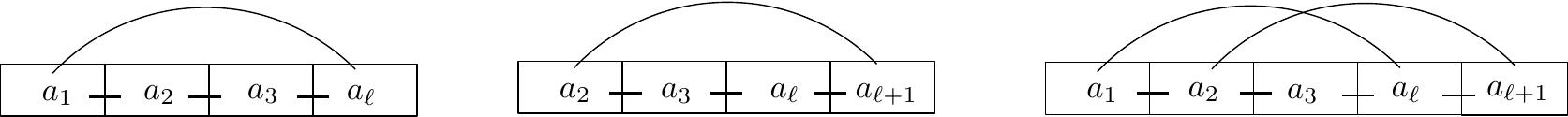}
  \end{center}
  \vskip-.5cm
  \caption{Sequences representing two vertices of $CK(d,4)$ and the arc between them. The lines drawn between characters indicate that these characters must be different.}\label{fig:verticesIteradedLineGraph}
\end{figure}

\vskip.5cm

\begin{figure}[t]
  \begin{center}
  \includegraphics[scale=0.8]{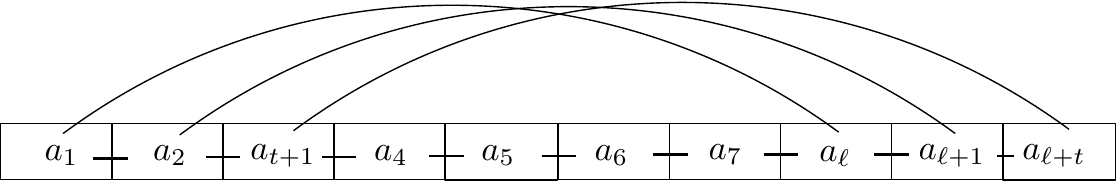}
  \end{center}
  \vskip-.5cm
  \caption{A sequence of $\ell+t$ characters that represents a vertex of the iterated line digraph of a cyclic Kautz digraph $CK(d,\ell=8)$ at iteration $t=2.$}\label{fig:sequencesS}
\end{figure}

\begin{lemma}\label{lemma:main}
The system
\begin{eqnarray*}
  B_\ell &=& (d-1)D_{\ell-2}+d\,C_{\ell-2} \\
  C_\ell &=& (d-1)D_{\ell-2}+d\,B_{\ell-2} \\
  D_\ell &=& (d^2-3d+3)D_{\ell-2}+(d-1)^2C_{\ell-2}+(d-1)^2B_{\ell-2}
\end{eqnarray*}
with initial values
\begin{eqnarray*}
  B_6 &=& (d+1)d(d-1)^3 \\
  C_6 &=& (d+1)d(d^3-2d^2+3d-1) \\
  D_6 &=& (d+1)d(d-1)^2(d^2-2d+3)
\end{eqnarray*}
has solution
\begin{eqnarray*}
B_\ell &=& d(d^2-d+1)^{\ell/2-1}+\frac{1}{2}(-1)^{\ell/2-1}d(d-1)(d-2)^{\ell/2-1}+\frac{1}{2}(-1)^{\ell/2}(d+1)d^{\ell/2},\\
C_\ell &=& \frac{1}{2}(-1)^{\ell/2-1}(d-1)d(d-2)^{\ell/2-1}+d(d^2-d+1)^{\ell/2-1}-\frac{1}{2}(-1)^{\ell/2}d^{\ell/2}(d+1), \\
D_\ell &=& (d-1)d\left((d^2-d+1)^{\ell/2 -1}-(-1)^{\ell/2-1}(d-2)^{\ell/2-1}\right).
\end{eqnarray*}
\end{lemma}

We defer the proof of this lemma to the Appendix.


\begin{figure}[t]
  \begin{center}
  \includegraphics[scale=0.8]{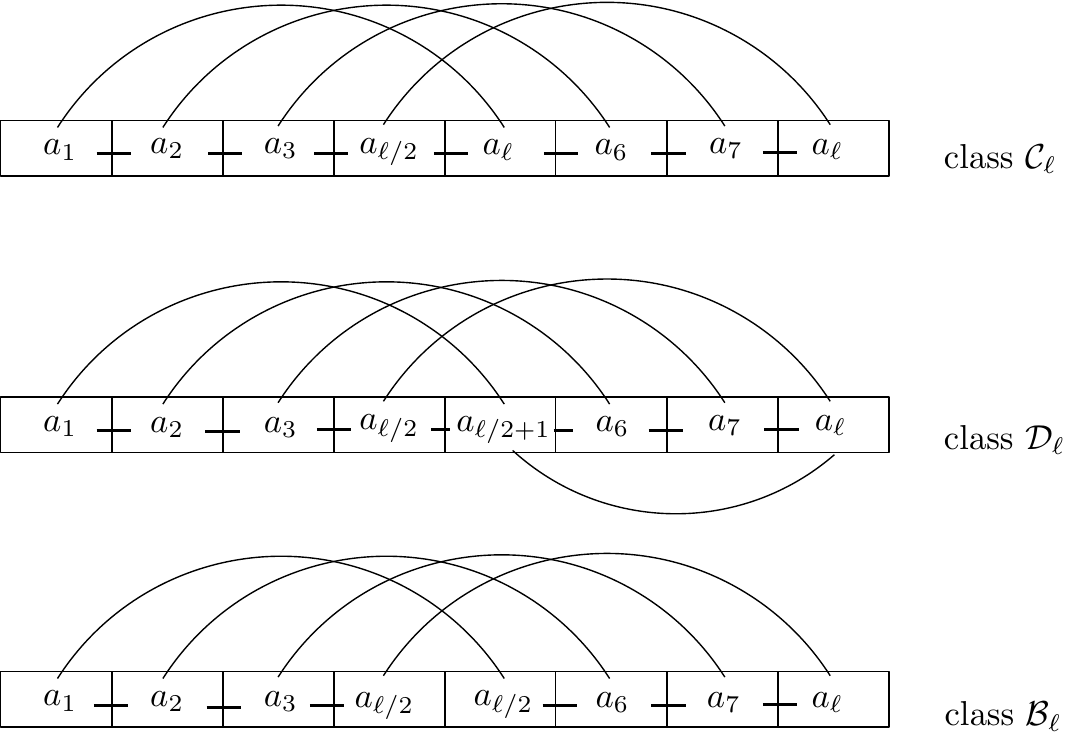}
  \end{center}
  \vskip-.5cm
  \caption{The classes $\mathcal{B}_\ell, \mathcal{C}_\ell,$ and $\mathcal{D}_\ell$.}\label{fig:classesBCD}
\end{figure}

\vskip.5cm

\begin{figure}[t]
  \begin{center}
  \includegraphics[scale=0.8]{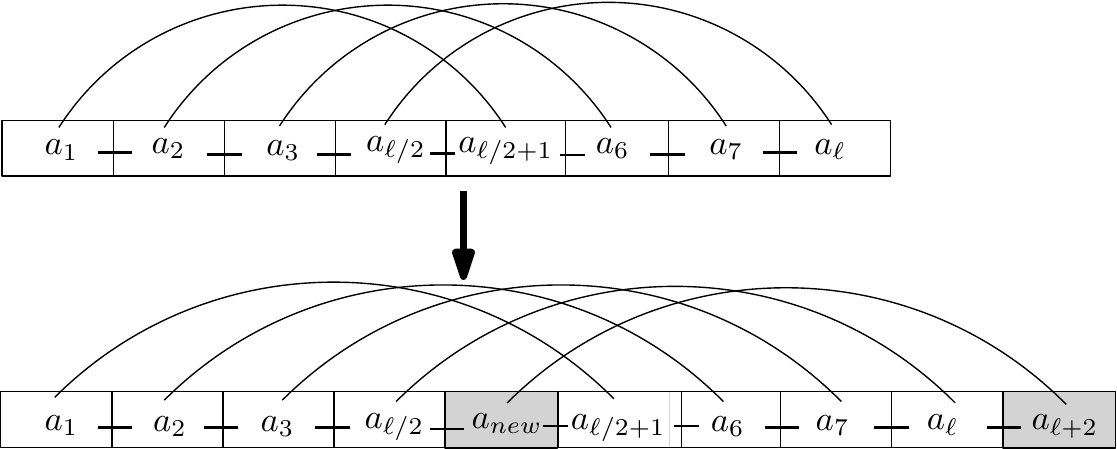}
  \end{center}
  \vskip-.5cm
  \caption{A sequence of length $\ell+2$ is obtained from one of length $\ell$ by inserting a new character between $a_{\ell/2}$ and  $a_{\ell/2+1}$ and another one, $a_{\ell+2}$, after $a_\ell$.}\label{fig:insertBCD}
\end{figure}

\begin{remark}
From Lemma~\ref{lemma:main} we immediately obtain the number of vertices of the $t=(\ell-2)$-iteration of the line digraph of a cyclic Kautz digraph $CK(d,\ell)$ for $d\geq 2, \ell \geq 3$, which is $C_{2\ell-2}+D_{2\ell-2}.$
\end{remark}


We now count the number of sequences from the set $\mathcal{S}$ in general.
Recall that an element of $\mathcal{S}$ has the form $a_1 a_2\ldots a_{\ell+t}$ satisfying $a_i \neq a_{i+1}$ for $1\leq i \leq \ell+t-1$ and
$a_i \neq a_{i+\ell-1}$ for $1\leq i\leq t+1.$ In the following, we set $r=t+1$ (we assume $r$ and $d$ are fixed integers) and represent the set of sequences $\mathcal{S}$ by $\mathcal{E}_j$, where $j=0,1,2,\ldots.$ An element of $\mathcal{E}_j$ has the form $a_1 \ldots a_{2r+j}$ satisfying $a_i \neq a_{i+1}$ for $1\leq i \leq 2r+j-1$ and
$a_i \neq a_{i+r+j}$ for $1\leq i\leq r.$ $E_j$ is the cardinality of  $\mathcal{E}_j$.
Observe that $E_0=C_{2r}+D_{2r}.$
To determine $\mathcal{E}_1$ we insert a new character between characters $a_r$ and $a_{r+1}$ in each sequence of $\mathcal{B}_{2r}, \mathcal{C}_{2r}$ and $\mathcal{D}_{2r}.$  This gives
$$E_1=d\,B_{2r} + (d-1)(C_{2r}+D_{2r}).$$
For $j>1,$ $\mathcal{E}_j$ is obtained from $\mathcal{E}_{j-1}$ and $\mathcal{E}_{j-2}.$
In each sequence of the set $\mathcal{E}_{j-1}$, we insert a new character between the characters $a_r$ and $a_{r+1}$, which can be done in $d-1$ ways.
In each sequence of the set $\mathcal{E}_{j-2}$, we first duplicate the character $a_r$, and then insert a new character between these two characters $a_r.$ This can be done in $d-1$ ways. Figure~\ref{fig:generateEj} depicts the insertion procedure.

\begin{figure}[t]
\vskip-3.5cm
  \begin{center}
  \includegraphics[scale=0.8]{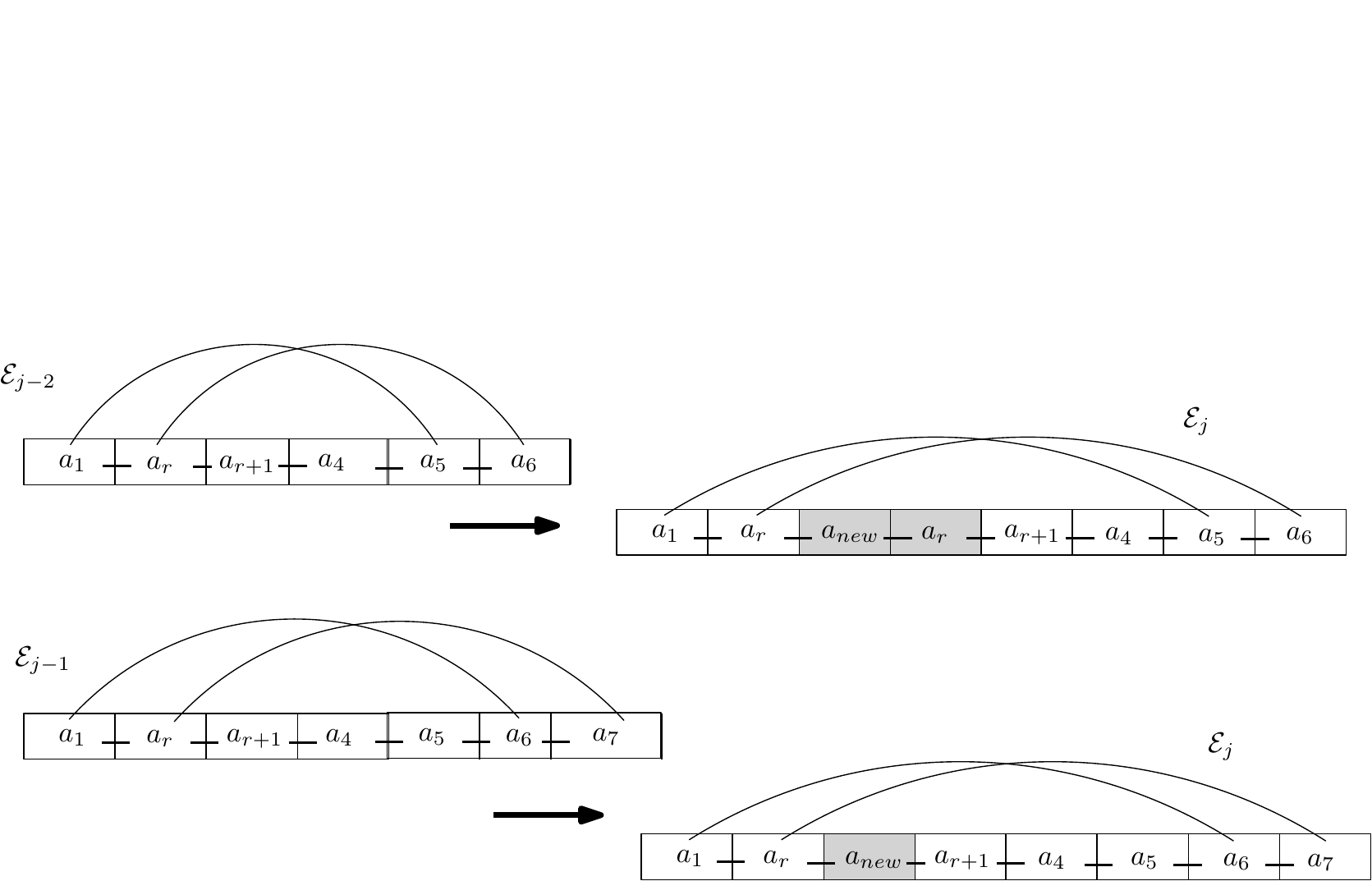}
  \end{center}
  \vskip-.5cm
  \caption{The sequences of $\mathcal{E}_j$ are obtained from the sequences of $\mathcal{E}_{j-1}$ and $\mathcal{E}_{j-2}.$ Here, $r=2$ and $j=4.$}\label{fig:generateEj}
\end{figure}

Note that each sequence of $\mathcal{E}_j$ is generated exactly once.
Thus, we only need to solve the recursion given in the following result.
\begin{lemma}\label{lem:recursion2}
For fixed integers $r$ and $d$, the recursion
$$E_j=(d-1)E_{j-1} + d\,E_{j-2}$$
with initial values $E_0=C_{2r}+D_{2r}$ and $E_1=d\,B_{2r} + (d-1)(C_{2r}+D_{2r})$
has solution
$$E_j= (-1)^j(B_{2r}+C_{2r}+D_{2r})\frac{1-(-d)^{j+1}}{d+1} - B_{2r}(-1)^j.$$
\end{lemma}

We defer the proof of this lemma to the Appendix.\\

We now use $B_{2r}+C_{2r}+D_{2r}=(d+1)d(d^2-d+1)^{r-1}$

and $$B_{2r}= d(d^2-d+1)^{r-1}+\frac{1}{2}(-1)^{r-1}d(d-1)(d-2)^{r-1}+\frac{1}{2}(-1)^r(d+1)d^r.$$

Then,
$$
E_j = (d^2-d+1)^{r-1}d^{j+2}+\frac{1}{2}(-1)^{r+j}(d-2)^{r-1}(d-1)d +\frac{1}{2}(-1)^{r+j+1}d^{r}(d+1).
$$

From $E_j$ we now obtain the number of vertices of the iterated line graph of $CK(d,\ell)$ at iteration $t\geq1$. Since $t=r-1$ and $j=\ell-t-2$ we arrive at the claimed formula of Theorem~\ref{theo:num-v-line-dig}:
$$
(d^2-d+1)^t d^{\ell-t}+\frac{1}{2}(-1)^{\ell+1}(d-2)^t(d-1)d +\frac{1}{2}(-1)^{\ell} d^{t+1}(d+1).
$$

\begin{remark}
Theorem~\ref{theo:num-v-line-dig} also holds for $t=0$, if for the particular case $d=2$ the indeterminate form $(d-2)^t$ is defined as $1$.
\end{remark}

Note that, in general, the $t$-iterated line digraph of a cyclic Kautz digraph is neither a Kautz digraph, nor a cyclic Kautz digraph.
But if the length is $\ell=2$, then it is clear that $CK(d,2)$ (and all its iterated line digraphs) are Kautz digraphs.


%% file: appendix.tex
\noindent{\bf{Proof of Lemma~\ref{lem:diamd1}}
\begin{proof}
Since each vertex of $CK(d,1)$ consists of exactly one symbol, there
is an arc between every pair of vertices of $CK(d,1)$.
For even $\ell$, $CK(1,\ell \geq 2)$ contains exactly two
vertices, $u$ and $v$. For both vertices the two symbols of the
alphabet alternate. Clearly, both ${u v}$ and ${v u}$ are arcs of
$CK(1,\ell \geq 2)$.
For odd~$\ell$, there is no sequence that would correspond to a vertex
of $CK(1,\ell \geq 2)$.
\end{proof}


\noindent{\bf{Proof of Lemma~\ref{lem:diamd2}}
\begin{proof}
Let $u = {u_1 u_2}$ and $v = {v_1 v_2}$ be two vertices of $CK(d \geq
2,2)$. If $u_2$ and $v_1$ are the same symbol, then ${u v}$ is an
arc of $CK(d \geq 2,2)$.
Otherwise, $u_2$ and $v_1$ are different symbols, ${u_2 v_1}$ is
a vertex of $CK(d \geq 2,2)$, and there is a path from $u$ to $v$ of
length $2$.
On the other hand, there is no path of length $1$ from vertex ${1 0}$
to vertex ${2 0}$ in $CK(d \geq 2,2)$.
Thus, the diameter of $CK(d \geq 2,2)$ is $2$.
\end{proof}


\noindent{\bf{Proof of Lemma~\ref{lem:pathIfOperations}}
\begin{proof}
To prove the reverse implication, we show that by performing any of
the
two given operations on the disc of a vertex $u$, we obtain the disc
of a vertex that is reachable from $u$ in $CK(d,\ell)$.
First note that rotating the disc of a vertex $u
= {u_1 u_2 {\ldots} u_\ell}$ by one position yields the disc of
a vertex
$u' = {u_2 {\ldots} u_\ell u_1}$, and $(u,u')$ is clearly an arc of
$CK(d,\ell)$.
Thus, by rotating the disc of $u$ by $k$ positions, we obtain a vertex
$u^{(k)} = {u_{k+1} {\ldots} u_\ell u_1 {\ldots} u_k}$ and there is
a path from $u$ to $u^{(k)}$ in $CK(d,\ell)$.

Now let $v$ be a vertex of $CK(d,\ell)$ such that the disc of $v$ can
be obtained from the disc of $u = {u_1 {\ldots} u_\ell}$ by swapping
the $k$-th symbol in a valid way. That is,
$v = {u_1 {\ldots} u_{k-1} x u_{k+1} {\ldots} u_\ell}$, with $x$
different
from both $u_{k-1}$ and $u_{k+1}$.
Then, the disc of $v$ can be obtained from the disc of $u$ as follows:
1) Rotate the disc of $u$ by $k-1$ positions to obtain the disc of
some vertex $r$; 2) Swap the symbol ($u_k$ for $x$) in the marked
position in the disc of vertex $r$ and rotate the resulting disc by
one
position to obtain the disc of a vertex $s$ with marked symbol
$u_{k+1}$;
3) Rotate the disc of $s$ by $\ell-k$ positions to obtain the disc of
a vertex $t = v$.
Note that, using the argument above, there is a path from $u$ to $r$
in
$CK(d,\ell)$, and there is also a path from $s$ to $v$ in
$CK(d,\ell)$.
Moreover, $(r,s)$ is an arc of $CK(d,\ell)$.
Thus, there is a path from $u$ to $v$ in $CK(d,\ell)$.

To show the forward implication, first recall that for every two
vertices $u$, $u'$ such that $(u,u')$ is an arc of $CK(d,\ell)$, the
disc of $u'$ can be obtained from the disc of $u$ by swapping the
symbol at the marked position and rotating the disc by one step.
Therefore, whenever $u$, $v$ are two vertices of $CK(d,\ell)$ such
that
there is a path from $u$ to $v$ in $CK(d,\ell)$, then the disc of $v$
can be obtained from the disc of $u$
as follows.
Start with the disc of $u$, follow the arcs of an $uv$-path
$p_1=u,p_2,{\dots},p_k=v$ and, for each arc $(p_i,p_j)$,
obtain the disc of $p_j$ from the disc of $p_i$ by swapping the symbol
at the marked position of $p_i$ and rotating the resulting disc by one
step.
\end{proof}


\noindent{\bf{Proof of Lemma~\ref{lem:OpsNotChangeImprint}}
\begin{proof}
To prove the reverse implication, first observe that performing any of
the two operations does not
change the number of $+/-$ signs in the imprint.
Rotation clearly has no influence on the number of $+$ and $-$ signs
in the imprint.
The swap of a symbol $u_i$ can only be performed if $u_{i-1}
= u_{i+1}$. This follows from the fact that the symbol at position
$u_i$ must be distinct from both its neighbors, and if $u_{i-1} \not=
u_{i+1}$, then there is only one valid symbol left for $u_i$ and
thus it cannot be swapped.
However, if $u_{i-1} = u_{i+1}$, then the two pairs $(u_{i-1}, u_i)$,
and $(u_i, u_{i+1})$ contribute to the imprint by exactly one $+$
sign and one $-$ sign, and this remains true also after the swap of
the symbol $u_i$.

To show the forward implication, let us now consider two vertices
$u$, $v$ of $CK(2,\ell)$ such that
the numbers of $+$ and $-$ signs in the two imprints agree.
We will show that this implies that the disc of $v$ can be obtained
from the disc of $u$ by a sequence of the permitted operations as
follows.
Since there are only $3$ different symbols in $u$ and $v$, there is
a symbol that appears in both of $u$ and $v$, say $u_s = v_t$.
Let us rotate the disc of $u$ by $s-1$ steps and the disc of $v$ by
$t-1$ steps yielding the discs of vertices $u'$ and $v'$ such that
both have the symbol $u_s = v_t$ at the marked position, that is,
$u_1' = v_1'$.

Now observe that by swapping the symbol at a position $i$ in the
disc of a vertex $w$, we can change the imprint of $w$ in positions
$i-1$ and $i$ from $-+$ to $+-$ (or vice versa).
For example, if $sgn(w_{i-1},w_{i}) = -$ and $sgn(w_{i},w_{i+1}) = +$,
then $w_{i-1} = w_{i+1}$, and thus, we can swap the symbol $w_{i}$ for
the remaining symbol (different from $w_{i-1}$ and $w_i$).
Based on this observation, notice that by a sequence of
operations of swapping a symbol of the disc of $w
= {w_1 {\ldots} w_\ell}$, we can obtain the disc of $w'
= {w_1' {\ldots} w_\ell'}$ such that $w_1' = w_1$ and the imprint
$im(w')
= (+ {\ldots} + - {\ldots} -)$ contains all the $+$ signs
contiguously at the beginning and all the $-$ signs contiguously at
the
end.
Thus, in particular, we can transform by a sequence of swaps the
discs of $u'$ and $v'$ into the discs of $u''$ and $v''$ such that
$u_1'' = u_1' = v_1' = v_1''$, and both imprints $im(u'')$ and
$im(v'')$ are of the form $(+ {\ldots} + - {\ldots} -)$.
Since the imprints of $u$ and $v$ have the same number of $+/-$ signs
and the performed operations do not change it, it follows that
$im(u'') = im(v'')$.

Finally, observe that the imprint of a vertex $w$ together with the
first symbol of $w$ determines the vertex $w$ in $CK(2,\ell)$
uniquely.
Since $im(u'') = im(v'')$ and $u_1'' = v_1''$, we obtain that $u''
= v''$. Therefore, the disc of $v$ can be obtained from the disc of
$u$ by a sequence of the listed operations.
\end{proof}


\noindent{\bf{Proof of Lemma~\ref{lem:diam51}}

\begin{proof}
We use the same strategy as in Lemma~\ref{lem:diam41}.
We show that for any pair of vertices $u$, $v$ of $CK(d\geq 4,\ell
\geq 4)$, there
is always a path of length at most $2\ell - 2$ or $2\ell - 3$.
In particular, we show that there is either a sequence
$y = {y_1  y_2  {\ldots}  y_{\ell-2}}$ of $\ell-2$ symbols or
a sequence
$z = {z_1  z_2  {\ldots}  z_{\ell-3}}$ of $\ell-3$ symbols, such that
the
concatenation of $u$, one of $y$ or $z$, and $v$ forms
a sequence such that any contiguous subsequence of length $\ell$
is a vertex of $CK(d\geq 4,\ell \geq 4)$.

Given the vertices $u$ and $v$, we distinguish two cases depending on
whether the symbols $u_{\ell}$ and $v_1$ differ or not.
Let us first assume that $u_{\ell} \not= v_1$.
We construct the sequence $y$ of $\ell-2$ symbols one by one,
respecting the following restrictions. The symbol $y_1$ must differ
from the symbols $u_2$, $u_{\ell}$, and $v_2$. In general, $y_i$ must
differ from the $u_{i+1}$, $y_{i-1}$, and $v_{i+1}$. Finally,
$y_{\ell-2}$ must also differ from $v_1$.
Since for each of the $y_i$ we have at most $4$ restrictions, and the
alphabet contains $5$ symbols, we can always find a symbol to assign
to $y_i$.
These restrictions ensure that the constructed sequence $y$ satisfies
that each subsequence of $\ell$ symbols of $u$,$y$ or of $y$,$v$ forms
a vertex of $CK(d\geq 4,\ell \geq 4)$.
Moreover, the assumption that $u_{\ell} \not= v_1$ ensures that also
the sequence ${u_{\ell} y_1 {\ldots} y_{\ell-2} v_1}$ forms a vertex
of
$CK(d\geq 4,\ell \geq 4)$.

Let us now assume that $u_{\ell} = v_1$.
This implies that no sequence $y$ of length $\ell-2$ satisfying the
above property can be found, since ${u_{\ell} y_1 {\ldots} y_{\ell-2}
v_1}$ is not a vertex of $CK(d\geq 4,\ell \geq 4)$.
However, this also implies that $u_{\ell} = v_1 \not= v_2$ and that
$v_1 = u_{\ell} \not= u_{\ell-1}$.
We construct the sequence $z$ of $\ell-3$ symbols one by one,
respecting analogous restrictions as in the previous case.
Again, due to the fact that the alphabet contains $5$ symbols, we can
find $z$, such that the subsequences of $\ell$ symbols of $u$, $z$ or
of $z$, $v$ form a vertex of $CK(d\geq 4,\ell \geq 4)$.
Moreover, since $u_{\ell} \not= v_1$, we observe that $u_{\ell} \not=
v_2$ and $v_1 \not= u_{\ell-1}$. Therefore also both the sequences
${u_{\ell-1} u_{\ell} z_1 {\ldots} z_{\ell-3} v_1}$ and
${u_{\ell} z_1 {\ldots} z_{\ell-3} v_1 v_2}$ form vertices of
$CK(d\geq
4,\ell \geq 4)$.

Therefore, we can either find a valid sequence $y$ of length $\ell-2$,
or a valid sequence $z$ of length $\ell-3$. This gives us an upper
bound of $2\ell-2$ on the diameter of $CK(d\geq 4,\ell \geq 4)$.
\end{proof}


\noindent{\textbf{Proof of Lemma~\ref{lem:diam_lb}}

\begin{proof}
$(a)$ ($d\geq 3, \ell = 3$, lower bound) \label{lem:diam33}
The diameter of $CK(d \geq 3,3)$ is at least $5 = 2\ell - 1$.

Let $u = {0 1 2}$ and $v = {2 1 0}$ be two vertices of $CK(d \geq
3,3)$.
We argue that $u$ and $v$ are at the distance at least $5$ as follows.
Since the only symbol $2$ in $u$ is at the position $u_3$,
but $v_1 = 2$, at least $2$ steps are needed to reach $v$ from $u$.
Since $u_2 = v_2$, then ${u_2 u_3 = v_1 v_2}$ is not a vertex
and $v$ cannot be reached from $u$ in $2$ steps.
Also, since $u_3 = v_1$ and ${u_3 v_1 v_2}$ is not a vertex,
thus also $3$ steps are not enough to reach $v$ from $u$.
Finally, since $u_3 = v_1$ and ${u_3 x v_1}$ is not a vertex for
any symbol $x$, thus $v$ cannot be reached from $u$ in $4$ steps.


$(b)$ ($d=3, \ell \geq 3$, lower bound) \label{lem:diam42}
The diameter of $CK(3,\ell \geq 3)$ is at least $2\ell - 1$.

We distinguish two cases depending on the parity of $\ell$, and in both
we describe a pair of vertices $u$ and $v$ of $CK(3,\ell \geq 3)$ that
are at the distance at least $2\ell - 1$.

First, let $\ell$ be odd.
Consider $u = {0 1 0 1 {\ldots} 0 1 2}$, and $v = {2 1 0 {\ldots} 1 0 1 0}$
of $CK(3,\ell \geq 3)$.
Since the only symbol $2$ in $u$ is at the position $u_\ell$, but $v$
contains $2$ at the position $v_1$, at least $\ell - 1$ steps are
needed to reach $v$ from $u$.
However, $\ell - 1$ steps are also not enough, since
$u_{\ell - 1} = 1 = v_{\ell - 1}$, and thus the sequence
${u_{\ell - 1} u_\ell = v_1 {\ldots} v_{\ell - 1}}$ is not a vertex
of $CK(3,\ell \geq 3)$.
For the sake of contradiction, assume that there is a path of length
$\ell + z$, for some $0 \leq z < \ell-1$.
Then, there must exist a sequence of $z$ symbols $x = (x_1, x_2, {\dots},
x_z)$, such that any contiguous subsequence of length $\ell$
of $u,x,v$ forms a vertex of $CK(3,\ell \geq 3)$.
For $z = \ell - 2$ no such sequence $x$ exists, since $u_{\ell}
= 2 = v_1$, which implies that ${u_{\ell} x_1 {\ldots} x_z v_1}$ is not
a vertex of $CK(3,\ell \geq 3)$.
If $z$ is odd, there is no such sequence $x$, since
$u_{\ell-1} = 1 = v_{\ell-2-z}$, which implies that
${u_{\ell-1} u_\ell x_1 {\ldots} x_z v_1 {\ldots} v_{\ell-2-z}}$ is not
a vertex of $CK(3,\ell \geq 3)$.
If $z$ is even, assume for the sake of contradiction that there is
such a sequence $x$. First observe that $x$ cannot contain any symbol
$0$ or $1$ as follows.
Since any subsequence of $u,x,v$ of $\ell$ consecutive symbols must
form a vertex of $CK(3,\ell \geq 3)$, the symbol at position $x_i$
must differ from the symbols at the positions $u_{i+1}$ and
$v_{\ell-1-z+i}$.
Thus, at position $x_i$, cannot be $0$: If $i$ is odd, also
$\ell-1-z+i$ is odd and thus $v_{\ell-1-z+i} = 0$ and otherwise if $i$
is even, then $i+1$ is odd and thus $u_{i+1} = 0$
(note that $0 \leq z < \ell-1$ implies $\ell-1-z+i > 1$, and $i \leq
z$, thus also $i+1 \leq \ell-1$).
Similarly, the symbol at the position $x_i$ cannot be $1$: If $i$ is
odd, $i+1$ is even and $u_{i+1} = 1$; otherwise, if $i$ is even,
$\ell-1-z+i$ is even and $v_{\ell-1-z+i} = 1$.
Therefore, since the size of the alphabet is $4$, the sequence $x$
must consists only of $2$'s and $3$'s.
Since $u_{\ell} = 2 = v_1$, both $x_1$ and $x_z$ must be $3$.
Then, $x_2 = x_{z-1} = 2$, and $x_3 = x_{z-2} = 3$, {\dots}
Since $z$ is even, the sequence $x$ has even length, and thus
$x_{z/2} = x_{z/2 + 1}$ (the same symbols meet in the middle of $x$),
which is not possible and we get a contradiction.
\smallskip

Now, let $\ell$ be even and let us proceed similarly.
Consider $u = {1 0 2 0 2 0 {\ldots} 2 0 1 2}$, and $v
= {2 1 3 0 2 0 {\ldots} 2 0 1 0}$ of $CK(3,\ell > 3)$ (for example, for
$\ell=4$, $u = {2 0 1 2}$, and $v = {2 1 3 0}$).
Since the vertex $v$ starts with the symbols $2,1$, but the vertex $u$
does not contain this pattern, at least $\ell-1$ steps are needed to
reach $v$ from $u$.
However, $\ell - 1$ steps are also not enough, since
$u_{\ell-1} = 1 = v_{\ell-1}$, and thus the sequence
${u_{\ell-1} u_\ell = v_1 {\ldots} v_{\ell-1}}$ is not a vertex
of $CK(3,\ell > 3)$.
Also $\ell$ steps do not suffice, since $u_{\ell} = 2 = v_1$, and
thus ${u_\ell v_1 {\ldots} v_{\ell-1}}$ is not a vertex
of $CK(3,\ell > 3)$.
Again, assume for the sake of contradiction that there is a path of
length $\ell + z$ connecting $u$ and $v$, for some $1 \leq
z < \ell-1$.
Then, there must exist a sequence of $z$ symbols $x = (x_1, x_2, {\dots},
x_z)$ such that any contiguous subsequence of length $\ell$
of $u,x,v$ forms a vertex of $CK(3,\ell > 3)$.
We distinguish two cases depending on the parity of~$z$.
If $z$ is odd, we distinguish 3 cases depending on $z$ (the length of
$x$), and for each we argue why there cannot be such a sequence $x$.
\begin{itemize}
\item $z < \ell - 5$: Independently of $x$, the subsequence
${u_{\ell-2} u_{\ell-1} u_{\ell} x_1 {\ldots} x_z v_1 {\ldots} v_{\ell-3-z}}$
is not a vertex of $CK(3,\ell > 3)$, since $u_{\ell-2}
= 0 = v_{\ell-3-z}$.
\item $z = \ell - 5$: Independently of $x$, the subsequence
${u_{\ell-3} {\ldots} u_{\ell} x_1 {\ldots} x_z v_1}$
is not a vertex of $CK(3,\ell > 3)$, since $u_{\ell-3} = 2 = v_{1}$.
\item $z = \ell - 3$:
We observe that there are the following restrictions on the symbols in
$x$: Since for each $i$, both ${u_{i+1} {\ldots} u_\ell x_1 {\ldots} x_i}$
and ${x_i {\ldots} x_z v_1 {\ldots} v_{i+2}}$ must be vertices of
$CK(3,\ell > 3)$, the symbol $x_i$ must differ from both $u_{i+1}$ and
$v_{i+2}$. Moreover, $x_1$ must differ from $u_\ell$, and $x_z$
must differ from $v_1$.
Considering $u$ and $v$, this implies that $x$ can only consist of
symbols $1$ and $3$, with $x_1 = 1$ and $x_z = 3$. However, since the
two symbols in $x$ must alternate, and $z$ is odd, all these
conditions cannot be satisfied and we get a contradiction.
\end{itemize}
If $z$ is even, we distinguish 3 cases depending on $z$ (the length of
$x$) and for each we identify a subsequence of $u,x,v$ of length
$\ell$ that does not form a vertex of $CK(3,\ell > 3)$:
\begin{itemize}
\item $z < \ell - 4$:
${u_\ell x_1 {\ldots} x_z v_1 {\ldots} v_{\ell-1-z}}$ is not a vertex,
since $u_\ell = 2 = v_{\ell-1-z}$.
\item $z = \ell - 4$:
${u_{\ell-1} u_\ell x_1 {\ldots} x_z v_1 v_2}$ is not a vertex,
since $u_{\ell-1} = 1 = v_2$.
\item $z = \ell - 2$:
${u_\ell x_1 {\ldots} x_z v_1}$ is not a vertex, since $u_\ell
= 2 = v_1$.
\end{itemize}

Therefore, in both cases (odd and even $\ell$), we identified a pair of
vertices of $CK(3,\ell \geq 3)$ that are at the distance at least
$2\ell - 1$, thus bounding the diameter of $CK(3,\ell \geq 3)$ from
below.

\vskip.25cm 

$(c)$ ($d \geq 4, \ell \geq 3$, lower bound) \label{lem:diam52}
The diameter of $CK(d\geq 4,\ell \geq 3)$ is at least $2\ell - 2$.

Consider $u = {{\ldots}  0 1 0 1 0 1 2}$ ($u$ begins with ${0 1}$ if
$\ell$ is odd and with ${1 0}$ if $\ell$ is even),
and $v = {1 3 2 0 2 0 2 {\ldots}}$
of $CK(d\geq 4,\ell \geq 3)$.
We show that the shortest path from $u$ to $v$ in $CK(d\geq 4,\ell
\geq 3)$
has length at least $2\ell - 2$ as follows.
Since the vertex $v$ starts with the symbols $1,3$, but the vertex $u$
does not contain this pattern, at least $\ell-1$ steps are needed to
reach $v$ from $u$.
Also, since $u_\ell \not= v_1$, $\ell-1$ steps are not enough.
For the sake of contradiction assume that there is a path of length
$\ell + z$, for some $0 \leq z < \ell-2$.
Then, there must exist a sequence of $z$ symbols $x = (x_1, x_2, {\dots},
x_z)$ such that any contiguous subsequence of length $\ell$
of $u,x,v$ forms a vertex of $CK(d\geq 4,\ell \geq 3)$.
We distinguish 4 cases depending on the parity of $z$ and $\ell$, and
for each we identify a subsequence of $u,x,v$ of length $\ell$ that
does not form a vertex of $CK(d\geq 4,\ell \geq 3)$:
%
\begin{itemize}
\item $z$ odd, $\ell$ odd:
${u_\ell x_1 {\ldots} x_z v_1 {\ldots} v_{\ell-1-z}}$ is not a vertex,
since $u_\ell = 2 = v_{\ell-1-z}$.
\item $z$ odd, $\ell$ even:
${u_{2+z} {\ldots} u_\ell x_1 {\ldots} x_z v_1}$ is not a vertex,
since $u_{2+z} = 1 = v_1$.
\item $z$ even, $\ell$ odd:
${u_{2+z} {\ldots} u_\ell x_1 {\ldots} x_z v_1}$ is not a vertex,
since $u_{2+z} = 1 = v_1$.
\item $z$ even, $\ell$ even:
${u_\ell x_1 {\ldots} x_z v_1 {\ldots} v_{\ell-1-z}}$ is not a vertex,
since $u_\ell = 2 = v_{\ell-1-z}$.
\end{itemize}
Therefore, the shortest path from $u$ to $v$ has length at least
$2\ell-2$.
\end{proof}


\noindent{\bf{Proof of Lemma~\ref{lemma:main}
\begin{proof}
Let $X_\ell=(B_\ell,C_\ell,D_\ell)^\top$. We want to solve the matrix recurrence equation $X_\ell=A\,X_{\ell-2}$, where
$$
A=\left(
  \begin{array}{ccc}
    0 & d & d-1 \\
    d & 0 & d-1 \\
    (d-1)^2 & (d-1)^2 & d^2-3d+3 \\
  \end{array}
\right).
$$
By induction, $X_\ell=A^{\ell/2-3}\,X_6$, with

$$
X_6=\Big((d+1)d(d-1)^3, (d+1)d(d^3-2d^2+3d-1),(d+1)d(d-1)^2(d^2-2d+3)\Big)^\top
$$
and
\begin{eqnarray*}
&& \hskip-.5cm A^{\ell/2-3} = P\,D^{\ell/2-3}\,P^{-1}=\\
&& \hskip-.5cm \left(
  \begin{array}{ccc}
    1 & \frac{1}{2} & \frac{1}{1-d} \\
    -1 & \frac{1}{2} & \frac{1}{1-d} \\
    0 & -1 & -1 \\
  \end{array}
\right)\left(
  \begin{array}{ccc}
    (-d)^{\ell/2-3} & 0 & 0 \\
    0 & (2-d)^{\ell/2-3} & 0 \\
    0 & 0 & (d^2-d+1)^{\ell/2-3} \\
  \end{array}
\right)\left(
  \begin{array}{ccc}
    1 & \frac{1}{2} & \frac{1}{1-d} \\
    -1 & \frac{1}{2} & \frac{1}{1-d} \\
    0 & -1 & -1 \\
  \end{array}
\right)^{-1},
\end{eqnarray*}
where $D$ and $P$ are, respectively, the matrices of eigenvalues and eigenvectors of $A$. From $X_\ell=A^{\ell/2-3}\,X_6$, we obtain
the claimed solution.
\end{proof}


\noindent{\bf{Proof of Lemma~\ref{lem:recursion2}}}

\begin{proof}
{\it{Claim:}} $E_i + E_{i-1} = d^i (B_{2r}+C_{2r}+D_{2r})$ holds for every $i>0$.
We proceed by induction on $i.$ For the base case $i=1$, we have
$E_1+E_0
= d(B_{2r}+C_{2r}+D_{2r}).$

For the induction step, we can assume that $E_i + E_{i-1} = d^i (B_{2r}+C_{2r}+D_{2r})$ and
show that the statement also holds for $i+1$.

$E_{i+1}=(d-1)E_{i}+d\,E_{i-1}=d(E_i+E_{i-1})-E_i.$

Then, $$E_{i+1}+E_i = d(E_i+E_{i-1}) = d d^{i} (B_{2r}+C_{2r}+D_{2r}) = d^{i+1}(B_{2r}+C_{2r}+D_{2r}).$$
{\it{End of claim.}}

The recursion can now be solved by substituting
$$E_j=d^{j}(B_{2r}+C_{2r}+D_{2r})-E_{j-1}= d^{j}(B_{2r}+C_{2r}+D_{2r})-d^{j-1}(B_{2r}+C_{2r}+D_{2r})+E_{j-2}$$
$$=\cdots=\left( \sum_{i=0}^{j} (B_{2r}+C_{2r}+D_{2r})d^i (-1)^{j+i} \right) -B_{2r}(-1)^j.$$
Note that the term $B_{2r}(-1)^{j}$ is subtracted because it is not present in the expression for $E_0$. This sum, a geometric series, is further simplified as
$$E_j=\left( (-1)^j(B_{2r}+C_{2r}+D_{2r})\sum_{i=0}^{j} (-d)^i  \right) -B_{2r}(-1)^j$$
$$= \left((-1)^j(B_{2r}+C_{2r}+D_{2r})\frac{1-(-d)^{j+1}}{d+1}\right) - B_{2r}(-1)^j.$$
\end{proof}